\documentclass[10pt,twocolumn]{IEEEtran}
\usepackage{amsmath,amssymb,amsfonts}
\usepackage{amsmath,amssymb,amsthm}
\usepackage{algorithmic}
\usepackage{graphicx,subfig,float,epstopdf}
\usepackage{float}
\usepackage{textcomp}
\usepackage{bm}
\usepackage{color}
\allowdisplaybreaks
\usepackage{amsmath}
\usepackage{amsfonts}
\usepackage{amssymb}
\usepackage{algorithm}
\usepackage{algorithmic}
\usepackage{graphicx}
\usepackage{verbatim}
\usepackage{amsthm}
\usepackage{xcolor}
\usepackage{multirow}

\newtheorem{theorem}{Theorem}
\newtheorem{lemma}{Lemma}
\newtheorem{remark}{Remark}

\newtheorem{example}{Example}
\newtheorem{definition}{Definition}

\newcommand{\beq}{\begin{equation}}
\newcommand{\eeq}{\end{equation}}
\newcommand{\beqnn}{\begin{equation*}}
\newcommand{\eeqnn}{\end{equation*}}
\newcommand{\beqy}{\begin{eqnarray}}
\newcommand{\eeqy}{\end{eqnarray}}
\newcommand{\beqynn}{\begin{eqnarray*}}
\newcommand{\eeqynn}{\end{eqnarray*}}
\newcommand{\bit}{\begin{itemize}}
\newcommand{\eit}{\end{itemize}}
\newcommand{\ben}{\begin{enumerate}}
\newcommand{\een}{\end{enumerate}}
\newcommand{\bex}{\begin{example}}
\newcommand{\eex}{\end{example}}

\newcommand{\balg}[1]{\begin{algorithm} \caption{#1}}
\newcommand{\ealg}{\end{algorithm}}

\newcommand{\balgc}{\begin{algorithmic}[1]}
\newcommand{\ealgc}{\end{algorithmic}}

\newcommand{\bary}{\begin{array}}
\newcommand{\eary}{\end{array}}
\newcommand{\bmx}{\begin{bmatrix}}
\newcommand{\emx}{\end{bmatrix}}
\newcommand{\bsmx}{\left[\begin{smallmatrix}}
\newcommand{\esmx}{\end{smallmatrix}\right]}
\newcommand{\bmxc}[1]{\left[\begin{array}{@{}#1@{}}}
\newcommand{\emxc}{\end{array}\right]}
\newcommand{\bcn}{\begin{center}}
\newcommand{\ecn}{\end{center}}

\newcommand{\R}{\boldsymbol{R}}

\begin{document}
\title{\Large\bf Recovery of  Signal and Image with Impulsive Noise via $\ell_1-\alpha \ell_2$ Minimization}

\author{Peng~Li, Huanmin~Ge and Wengu~Chen
 \thanks{P. ~Li is with  Graduate School, China Academy of Engineering Physics, Beijing 100088, China (E-mail: lipeng16@gscaep.ac.cn)}
 \thanks{ H. ~Ge is with Sports Engineering College, Beijing Sport University,
Beijing 100084, China (E-mail: gehuanmin@163.com).}
\thanks{W.~Chen is with Institute of Applied Physics and Computational Mathematics,
Beijing 100088, China (E-mail: chenwg@iapcm.ac.cn).  } }
\date{}

\maketitle

\begin{abstract}
In this paper, we consider the efficient and robust  reconstruction of  signals and images
 via $\ell_{1}-\alpha \ell_{2}~(0<\alpha\leq 1)$ minimization in impulsive noise case.
 To achieve this goal, we introduce  two new models: the $\ell_1-\alpha\ell_2$ minimization with $\ell_1$ constraint, which is called $\ell_1-\alpha\ell_2$-LAD, the $\ell_1-\alpha\ell_2$ minimization with Dantzig selector constraint, which is called $\ell_1-\alpha\ell_2$-DS.
 We first  show that sparse signals or nearly sparse signals can be exactly or stably recovered  via $\ell_{1}-\alpha\ell_{2}$ minimization under some conditions based on   the restricted $1$-isometry property ($\ell_1$-RIP). Second, for $\ell_1-\alpha\ell_2$-LAD model, we  introduce unconstrained $\ell_1-\alpha\ell_2$ minimization model  denoting  $\ell_1-\alpha\ell_2$-PLAD and propose $\ell_1-\alpha\ell_2$LA algorithm to solve   the $\ell_1-\alpha\ell_2$-PLAD.
 Last, numerical experiments %on success rates of sparse signal recovery
 demonstrate that when the sensing matrix is ill-conditioned (i.e., the coherence of
  the matrix is larger than 0.99),  the $\ell_1-\alpha\ell_2$LA  method
  is better than  the  existing convex and non-convex compressed sensing solvers for the  recovery of sparse signals.  And for the magnetic resonance imaging (MRI) reconstruction with impulsive noise,  we show that
  the $\ell_1-\alpha\ell_2$LA  method has   better performance  than state-of-the-art methods via numerical experiments.
\end{abstract}

\IEEEpeerreviewmaketitle
\section{Introduction}\label{s1}

\hskip\parindent
\IEEEPARstart{C}{ompressed}
 sensing predicts that sparse signals can be reconstructed from what was previously believed to be incomplete information.  Since Cand\`{e}s, Romberg and Tao's seminal works \cite{CRT2006,CRT2006-1} and Donoho's ground-breaking work \cite{D2006},  this new field has triggered a large research in mathematics, engineering and medical image. In this contexts, it aims  to recover an unknown signal $\bm{x}\in\mathbb{R}^n$ from an underdetermined system of linear equations
\begin{align}\label{systemequationsnoise}
\bm{b}=\bm{A}\bm{x}+\bm{z},
\end{align}
where $\bm{b}\in\mathbb{R}^m$ are available measurements, the matrix $\bm{A}\in\mathbb{R}^{m\times n}~(m<n)$ models the linear measurement process and $\bm{z}\in\mathbb{R}^m$ is a measurement errors.

For the reconstruction of $\bm{x}$, the most intuitive approach is to find the sparsest signal in the  set of feasible solutions, which leads to the $\ell_0$ minimization method as follows
\begin{align}\label{VectorL0}
\min_{\bm{x}\in\mathbb{R}^n}\|\bm{x}\|_0~~\text{subject~ to}~~\bm{b}-\bm{A}\bm{x}\in\mathcal{B},
\end{align}
where $\|\bm{x}\|_0$ (it usually is called the $\ell_0$ norm of $\bm{x}$, but is not
a norm) denotes  the number of nonzero coordinates, and $\mathcal{B}$ is a bounded set determined by the error structure. However, such method is NP-hard and thus computationally infeasible in high dimensional background.

Cand\`{e}s and Tao \cite{CT2005} proposed a convex relaxation of the $\ell_0$ minimization method$-$the constrained $\ell_1$ minimization method:%. It estimates the signal $\bm{x}$ by
\begin{align}\label{BPmodel}
\min_{\bm{x}\in\mathbb{R}^n}~\|\bm{x}\|_1~~\text{subject~ to}~~\bm{A}\bm{x}=\bm{b},
\end{align}
which is also called basis pursuit (BP) \cite{CDS1998}.  In noisy case, i.e., $\bm{z}\neq \bm{0}$,
the above method is generalized.
For example,  when $\|\bm{z}\|_2\leq\eta$ (the $\ell_2$ bounded noise),
%the signal $\bm{x}$ is estimated in
\cite{CRT2006,DET2006} proposed the following method:
\begin{align}\label{QCBPmodel}
\min_{\bm{x}\in\mathbb{R}^n}~\|\bm{x}\|_1 ~~\text{subject~ to}~~\|\bm{b}-\bm{A}\bm{x}\|_{2}\leq\eta
\end{align}
for some constant $\eta>0$, which is called quadratically constrained basis pursuit (QCBP). Instead of solving (\ref{QCBPmodel}) directly, many authors also studied the following unconstrained Lasso method \cite{T1996}:
\begin{align}\label{VectorL1-Lasso}
\min_{\bm{x}\in\mathbb{R}^n}~\lambda\|\bm{x}\|_1+\frac{1}{2}\|\bm{A}\bm{x}-\bm{b}\|_2^2,
\end{align}
where $\lambda> 0$ is a parameter to balance the data fidelity term $\|\bm{A}\bm{x}-\bm{b}\|_2^2/2$ and the objective function $\|\bm{x}\|_1$. %We point out this problem was first introduced in \cite{T1996}.
A large amount of literature on the $\ell_1$ minimization has emerged.
%However,  this paper aim to recover $\bm{x}$ from \eqref{systemequationsnoise} using
%different approaches from the $\ell_1$ minimization method.
%We shall not conduct a review here, as this paper is concerned with a
%different approach.

Some nonconvex relaxations of $\ell_0$ minimization as alternatives to convex relaxation $\ell_1$ minimization, which can give closer approximations to $\ell_0$,  promote sparsity better than $\ell_1$ minimization. The popular nonconvex relaxations  method include $\ell_p$  ($0<p<1$) minimization  and its variants \cite{C2007,CS2008,CWB2008,DG2009,DDF2010,S2011,S2012,XCXZ2012,LXY2013,WC2013,WLZ2015,ZL2017}
 and $\ell_{1}-\alpha\ell_{2}$ minimization in \cite{ELX2013,YEX2014,LYHX2015,YLHX2015,LZOX2015,YX2015,LYX2016,YSX2017,MLH2017,LP2017,LY2018,GWC2018}.
And in this paper, we only focus on $\ell_{1}-\alpha\ell_{2}$ minimization.

It is noted that \cite{ELX2013,YEX2014} focused on recovering nonnegative signal, i.e., $\bm{x}\geq \bm{0}$.
%There exists many works of $\ell_p$ minimization, reader can refer to \cite{DG2009,S2012,LXY2013,WC2013,WLZ2015,ZL2017}.
%In this paper, we focus on $\ell_{1-2}$-minimization.
And in this paper, we focus on recovering signal $\bm{x}\in\mathbb{R}^n$.
To recover $\bm{x}\in\mathbb{R}^n$, \cite{LP2017,LY2018} proposed  $\ell_{1}-\alpha \ell_{2}$ ($0<\alpha\leq 1$) minimization:
\begin{align}\label{VectorL1-alphaL2}
\min_{\bm{x}\in\mathbb{R}^n}~\|\bm{x}\|_{1}-\alpha\|\bm{x}\|_{2} ~~\text{subject~ to}~ ~\bm{b}-\bm{A}\bm{x}\in\mathcal{B}.
\end{align}
 When $\alpha=1$,  \eqref{VectorL1-alphaL2} reduces the $\ell_{1-2}$ minimization in \cite{LYHX2015,YLHX2015}.
Specifically,   Lou, et. al. in \cite{LYHX2015} considered the noiseless case $\mathcal{B}=\{0\}$, i.e.,
\begin{align}\label{VectorL1-L2andExact}
\min_{\bm{x}\in\mathbb{R}^n}~\|\bm{x}\|_{1}-\|\bm{x}\|_{2} ~~\text{subject~ to}~ \bm{A}\bm{x}=\bm{b}
\end{align}
and gave the restricted isometry property (RIP) characterization of this problem.  And they also proposed a DCA method to solve the unconstrained problem  corresponding to \eqref{VectorL1-L2andExact}, which is called $\ell_{1-2}$-Lasso:
\begin{align}\label{VectorL1-L2-Lasso}
\min_{\bm{x}\in\mathbb{R}^n}~\lambda(\|\bm{x}\|_{1}-\|\bm{x}\|_{2})+\frac{1}{2}\|\bm{A}\bm{x}-\bm{b}\|_2^2.
\end{align}
Yin, et.al. \cite{YLHX2015} considered the noisy case, i.e.,
\begin{align}\label{VectorL1-L2-DN}
\min_{x\in\mathbb{R}^n}~\|\bm{x}\|_{1}-\|\bm{x}\|_{2} ~~\text{subject~ to}~ \|\bm{A}\bm{x}-\bm{b}\|_2\leq\eta_1,
\end{align}
where $\eta_1\geq0$ is the noise level. The numerical examples in \cite{LYHX2015,YLHX2015} demonstrate that the $\ell_{1}-\ell_{2}$ minimization consistently outperforms the $\ell_1$ minimization and iterative strategies for $\ell_p$ minimization \cite{LXY2013} when the
measurement matrix $\bm{A}$ is highly coherent. In addition,  $\ell_{1-2}$ has shown advantages in various applications such as image restoration \cite{LZOX2015}, phase retrieval \cite{YX2015}, and point source super-resolution \cite{LYX2016} and uncertainty quantification \cite{YSX2017} and matrix completion \cite{MLH2017}.

 In order to deal with heavy tail and heteroscedastic noise,
\cite{YZ2011,W2013} proposed the $\ell_1$ penalized least absolute deviation ($\ell_1$-PLAD),  insteading of Lasso, i.e.,
\begin{align}\label{VectorL1-LAD}
\min_{\bm{x}\in\mathbb{R}^n}~\|\bm{A}\bm{x}-\bm{b}\|_1+\lambda~\|\bm{x}\|_1.
\end{align}
%which is used to deal with heavy tail and heteroscedastic noise.
Numerical examples in \cite{W2013} showed that the $\ell_1$-PLAD method \eqref{VectorL1-LAD}
 is better than the classical Lasso method \eqref{VectorL1-Lasso} for the heavy tail noise.

For  working with $\ell_p$ $(0<p\leq 1)$ norm,
Chartrand and Staneva \cite{CS2008} first proposed the restricted $p$ ($0<p\leq 1$)-isometry property ($\ell_p$-RIP), i.e.,
\begin{align}\label{Lq-RIP}
(1-\delta_s)\|\bm{x}\|_2^p\leq\|\bm{A}\bm{x}\|_p^p\leq(1+\delta_s)\|\bm{x}\|_2^p
\end{align}
for all $\bm{x}$ such that $\|\bm{x}\|_0\leq s$.
In \cite{CZ2015}, Cai and Zhang used the restricted 1-isometry property to characterize the exact and stable recovery of low-rank matrices.

Motivated by \cite{W2013,CZ2015,YLHX2015}, we will consider the $\ell_{1}-\alpha \ell_{2}$ minimization with $\ell_1$ constraint:
\begin{align}\label{VectorL1-alphaL2-LAD}
\min_{x\in\mathbb{R}^n}~\|\bm{x}\|_{1}-\alpha\|\bm{x}\|_{2} ~~\text{subject~ to}~ \|\bm{b}-\bm{A}\bm{x}\|_1\leq\eta_1
\end{align}
for some constant $\alpha\in(0,1]$ and $\eta_1\geq0$. The method is  called  $\ell_1-\alpha \ell_2$-LAD. In this paper,  we first give the $\ell_1$-RIP analysis for \eqref{VectorL1-alphaL2-LAD}.
Second, in order to solve  (\ref{VectorL1-alphaL2-LAD}),
 we present  the following unconstrained problem
corresponding to (\ref{VectorL1-alphaL2-LAD}):
\begin{align}\label{VectorL1-alphaL2-PLAD}
\min_{\bm{x}\in\mathbb{R}^n}~\lambda~(\|\bm{x}\|_1-\alpha\|\bm{x}\|_2)+\|\bm{A}\bm{x}-\bm{b}\|_1,
\end{align}
where $\lambda>0$ is a regularization parameter.  (\ref{VectorL1-alphaL2-PLAD}) is denoted $\ell_1-\alpha\ell_2$-PLAD. Next, we introduce a new algorithm   to compute proposed model (\ref{VectorL1-alphaL2-PLAD}). Last, numerical experiments are presented for the sparse signal and MRI image recovery problems.

The underdetermined problem (\ref{systemequationsnoise}) puts forward
both theoretical and computational challenges at the interface of statistics and
optimization (see, e.g., \cite{DET2006,MB2006,ZY2006}). In \cite{CT2007}, the so-called Dantzig selector was proposed
to perform variable selection and model fitting in the linear regression model. Its mathematical form is
\begin{align}\label{VectorL1-DS}
\min_{\bm{x}\in\mathbb{R}^n}~\|\bm{x}\|_{1}~~\text{subject~ to}~ \|\bm{A}^*(\bm{b}-\bm{A}\bm{x})\|_\infty\leq\eta_2
\end{align}
where $\eta_2\geq 0$ is a tuning or penalty parameter.
In \cite{CT2007}, performance of the Dantzig selector was analyzed theoretically by deriving
sharp nonasymptotic bounds on the error of estimated coefficients in the $\ell_2$ norm.

The Dantzig selector relates closely to Lasso (\ref{VectorL1-Lasso}). In some sense, Lasso estimator and Dantzig selector exhibit similar behavior. Essentially, the Dantzig selector model
(\ref{VectorL1-DS}) is a linear program while the Lasso model (\ref{VectorL1-Lasso}) is a quadratic program. They have the same objective function but with different constraints. For an extensive study on the relation between the
Dantzig selector and Lasso, we refer to a series of discussion papers which have been
published in The Annals of Statistics, e.g., \cite{B2007,CL2007,CT2007-1,EHT2007,FS2007,MRY2007,R2007}.

In this paper, we also consider $\ell_{1}-\alpha\ell_{2}$ minimization with Dantzig selector constraint
\begin{align}\label{VectorL1-alphaL2-DS}
\min_{\bm{x}\in\mathbb{R}^n}~\|\bm{x}\|_{1}-\alpha\|\bm{x}\|_{2} ~~\text{subject~ to}~ \|\bm{A}^*(\bm{b}-\bm{A}\bm{x})\|_\infty\leq\eta_2
\end{align}
for some constant $\eta_2\geq0$. We denote it as $\ell_{1}-\alpha\ell_{2}$-DS. Especially, when $\eta_1=0$ in (\ref{VectorL1-alphaL2-LAD}) or $\eta_2=0$ in (\ref{VectorL1-alphaL2-DS}), we consider
\begin{align}\label{VectorL1-alphaL2-Exact}
\min_{\bm{x}\in\mathbb{R}^n}~\|\bm{x}\|_{1}-\alpha\|\bm{x}\|_{2} ~~\text{subject~ to}~ \bm{A}\bm{x}=\bm{b}.
\end{align}

Besides establishing the $\ell_{1}$-RIP theory analysis, we also consider how to compute proposed model (\ref{VectorL1-alphaL2-PLAD}).
Combining ADMM \cite{BPCBJ2011} with DCA \cite{YLHX2015}, we propose an efficient algorithm $\ell_1-\alpha \ell_2$LA for $\ell_1-\alpha \ell_2$-PLAD problem (\ref{VectorL1-alphaL2-PLAD}).   Numerical experiments based on the $\ell_1-\alpha \ell_2$LA algorithm, for simulated signals and images show that the $\ell_1-\alpha \ell_2$LA algorithm is more robust than $\ell_1$-regularization based method and $\ell_p~(0<p<1)$-regularization based method. Our contributions of this paper can be stated as follows.

\begin{itemize}
\item[(1)]Two new models: $\ell_1-\alpha\ell_2$LAD and $\ell_1-\alpha\ell_2$-DS, are introduced, which are suitable for impulsive noise.
\item[(2)]In noiseless case, a uniform $\ell_1$-RIP condition for sparse signal recovery via  (\ref{VectorL1-alphaL2-Exact}) is established. In noisy case,
    the conditions based on $\ell_1$-RIP for the  recovery of  nearly sparse signals via $\ell_{1}-\alpha\ell_{2}$-LAD or $\ell_{1}-\alpha\ell_{2}$-DS are obtained, respectively.
\item[(3)]Combining ADMM \cite{BPCBJ2011} with DCA \cite{YLHX2015}, we propose $\ell_1-\alpha \ell_2$LA algorithm to compute $\ell_{1}-\alpha\ell_{2}$-PLAD model (\ref{VectorL1-alphaL2-PLAD}).
\item[(4)]We present performance analysis for sparse signal and compressible image recovery by numerical experiments based on the proposed $\ell_1-\alpha \ell_2$LA algorithm.
\end{itemize}

Throughout the article, we use the following basic notations. We denote $\mathbb{Z}_+$ by positive integer set. For any positive integer $n$, let $[[1,n]]$ denote the set $\{1,\ldots,n\}$. For $\bm{x}\in\mathbb{R}^n$,  denote $\bm{x}_{\max(s)}$ as the vector $\bm{x}$ with all but the largest $s$ entries in absolute value set to zero, and $\bm{x}_{-\max(s)}=\bm{x}-\bm{x}_{\max(s)}$. Let $\bm{x}_S$ be the vector equal to $\bm{x}$ on $S$ and to zero on $S^c$. Let $\|\bm{x}\|_{\alpha,1-2}$ denote $\|\bm{x}\|_1-\alpha\|\bm{x}\|_2$. And we denote $n\times n$ identity matrix by $\bm{I}_n$. And we denote the transpose of matrix $\bm{A}$ by $\bm{A}^*$. Use the phrase ``$s$-sparse vector" to refer to vectors of sparsity at most $s$. We use boldfaced letter denote matrix or vector.

%%%%%%%%%%%%%%%%%%%%%%%%%%%%%%%%%%%%%%%%%%%%%%%%%%%%%%%%%%%%%%%%%%%%
%%%%%%%%%%%%%%%%%%%%%%   section 2 %%%%%%%%%%%%%%%%%%%%%%%%%%%%%%%%%
%%%%%%%%%%%%%%%%%%%%%%%%%%%%%%%%%%%%%%%%%%%%%%%%%%%%%%%%%%%%%%%%%%%
\section{Exact Recovery via $\ell_{1}-\alpha\ell_{2}$ Minimization\label{s2}}\label{sec:exactrecovery}
\hskip\parindent

In this section, we will consider the exact recovery of $\bm{x}$ from \eqref{systemequationsnoise} via the method \eqref{VectorL1-alphaL2-Exact}. In order to characterize the exact recovery of $\bm{x}$,
 we first introduce the following definition of restricted $(\ell_2,\ell_p)$-isometry property.
\begin{definition}\label{LowerUpperlpRIP}
For $0<p\leq 1$, $s\in\mathbb{Z}_+$, we define the restricted $\ell_2/\ell_p$ isometry constant pair $(\delta_s^{lb},\delta_{s}^{ub})$ of order $s$ with respect to the measurement matrix $\bm{A}\in\mathbb{R}^{m\times n}$ as the smallest numbers $\delta_s^{lb}$ and $\delta_{s}^{ub}$ such that
\begin{align}\label{Vectorlq-RUB}
(1-\delta_s^{lb})\|\bm{x}\|_2^p\leq\|\bm{Ax}\|_p^p\leq(1+\delta_s^{ub})\|\bm{x}\|_2^p,
\end{align}
holds for all $s$-sparse signals $\bm{x}$.
We say that $\bm{A}$ satisfies the $(\ell_2,\ell_p)$-RIP  if $\delta_{s}^{lb}$ and $\delta_{s}^{ub}$ are small for reasonably large $s$.
\end{definition}

\begin{remark}\label{lqRIP-Remark}
When $\delta_s^{lb}=\delta_s^{ub}=\delta_s$,  Definition \ref{LowerUpperlpRIP} is the definition of
 the $\ell_p$-RIP (see \eqref{Lq-RIP}).
\end{remark}

\subsection{Auxiliary Lemmas \label{s2.1}}
\hskip\parindent

By the proof of \cite[Theorem 3.3]{YSX2017}, we have the following lemma, which is a modified cone
constraint inequality for $\ell_{1}-\alpha\ell_{2}$.
\begin{lemma}\label{ConeconstraintinequalityforL1-L2}
For any vectors $\bm{x}, \hat{\bm{x}}$, let $\bm{h}=\hat{\bm{x}}-\bm{x}$.
Assume that $\|\hat{\bm{x}}\|_{\alpha,1-2}\leq\|\bm{x}\|_{\alpha,1-2}$. Then
\begin{align}\label{e:h-maxsupperbound1}
&\|\bm{h}_{-\max(s)}\|_1\leq\|\bm{h}_{\max(s)}\|_1+2\|\bm{x}_{-\max(s)}\|_1+\alpha\|\bm{h}\|_2,\\
&\|\bm{h}_{-\max(s)}\|_1-\alpha\|\bm{h}_{-\max(s)}\|_2\leq\|\bm{h}_{\max(s)}\|_1+2\|\bm{x}_{-\max(s)}\|_1\nonumber\\
&\ \ \ \ \ \ \ \ \ \ \ \ \ \ \ \ \ \ \ \ \ \ \ \ \ \ \ \ \ \ \ \ \ \ \ \  +\alpha\|\bm{h}_{\max(s)}\|_2.
\label{e:h-maxsupperbound2}
\end{align}
Especially, when $\bm{x}$ is $s$-sparse, one has
\begin{align}\label{e:h-maxsupperboundnoiseless1}
\|\bm{h}_{-\max(s)}\|_1&\leq\|\bm{h}_{\max(s)}\|_1+\alpha\|\bm{h}\|_2,\\
\|\bm{h}_{-\max(s)}\|_1-\alpha\|\bm{h}_{-\max(s)}\|_2&\leq\|\bm{h}_{\max(s)}\|_1
+\alpha\|\bm{h}_{\max(s)}\|_2.
\label{e:h-maxsupperboundnoiseless2}
\end{align}
\end{lemma}

The following lemma is the fundamental properties of the function $\|\bm{x}\|_1-\alpha\|\bm{x}\|_2$ with
$0\leq\alpha\leq 1$, which  is a  generalization of  \cite[Lemma 2.1 (a)]{YLHX2015}.
It will be frequently used in our proofs.
\begin{lemma}\label{LocalEstimateL1-L2}
For $0 \leq \alpha\leq 1$,
suppose $\bm{x}\in\mathbb{R}^n\backslash\{\bm{0}\}$, $T=\text{supp}(\bm{x})$ and $\|\bm{x}\|_0=s$, then
\begin{align}\label{e:l12alphas}
(s-\alpha\sqrt{s})\min_{j\in T}|x_j|\leq\|\bm{x}\|_1-\alpha\|\bm{x}\|_2\leq(\sqrt{s}-\alpha)\|\bm{x}\|_2.
\end{align}
\end{lemma}

\begin{proof}
Without loss of generality,
let $|x_1|\geq |x_2|\cdots \geq |x_s|>|x_{s+1}|= \cdots =|x_n|=0$ and $t=\lfloor \sqrt{s}\rfloor$,
one has
\begin{align*}
\|\bm{x}\|_2\leq \sum_{i=1}^t|x_i|+(\sqrt{s}-t)|x_{t+1}|,
\end{align*}
which is \cite[(6.1)]{YLHX2015}.
Then
\begin{align}\label{e:l12alphalower}
&\|\bm{x}\|_1-\alpha\|\bm{x}\|_2
 \geq \|\bm{x}\|_1-\alpha(\sum_{i=1}^t|x_i|+(\sqrt{s}-t)|x_{t+1}|)\nonumber\\
&= [1-\alpha(\sqrt{s}-t)]|x_{t+1}|+\sum_{i=t+2}^s|x_i|+(1-\alpha)\sum_{i=1}^t|x_i|\nonumber\\
&\overset{(1)}{\geq}[1-\alpha(\sqrt{s}-t)]|x_{s}|+\sum_{i=t+2}^s|x_s|+(1-\alpha)\sum_{i=1}^t|x_s|\nonumber\\
&=(s-\alpha\sqrt{s})|x_{s}|\overset{(2)}{=} (s-\alpha\sqrt{s})\min_{i\in T}|x_i|,
\end{align}
where  (1) and (2) follow from $|x_1|\geq |x_2|\cdots \geq |x_s|$
and $0\leq\alpha\leq1 $.
\end{proof}

\begin{lemma}\label{LowerBound}
Assume that $\|\hat{\bm{x}}\|_{\alpha,1-2}\leq\|\bm{x}\|_{\alpha,1-2}$.
Let $\bm{h}=\hat{\bm{x}}-\bm{x}$, $T_0=\text{supp}(\bm{h}_{\max(s)})$,
$T_1$ be the index set of the $t\in\mathbb{Z}_+$
largest entries of $\bm{h}_{-\max(s)}$  and $T_{01}=T_0\cup T_1$,
the matrix $\bm{A}$ satisfies the  $(\ell_2,\ell_1)$-RIP condition
of $t+s$ order. Then
\begin{align}\label{e:Ahlowerbound}
\|\bm{Ah}\|_1
\geq&\rho_t\|\bm{h}_{T_{01}}\|_2
-\frac{2(1+\delta_{t}^{ub})\|\bm{x}_{-\max(s)}\|_1}{\sqrt{t}-\alpha},
\end{align}
where
\begin{align}\label{e:rho1}
\rho_{t}=1-\delta_{t+s}^{lb}-\frac{(1+\delta_t^{ub})}{a(s,t;\alpha)}
\end{align}
and $a(s,t;\alpha)=\frac{\sqrt{t}-\alpha}{\sqrt{s}+\alpha}$.
\end{lemma}
\begin{proof}
%Let $\text{supp}(h_{\max(s)})=:T_0$.
First, we partition $T_0^c=[[1,n]]\backslash T_0$  as
$$
T_0^c=\cup_{j=1}^J T_j,
$$
where $T_1$ is the index set of the $t$ largest entries of $\bm{h}_{-\max(s)}$, $T_2$ is the index set of the next $t$ largest entries of $\bm{h}_{-\max(s)}$, and so on.
%Here, $t\in\mathbb{Z}_+$ is to be determined.
Notice that the last index set $T_J$ may contain less $t$ elements.

By $T_{01}=T_0\cup T_1$, one has
%we consider the following identity
\begin{align}\label{e2.3}
\|\bm{Ah}\|_1=&\bigg\|\bm{A}\bm{h}_{T_{01}}+\sum_{j\geq 2}^J\bm{A}\bm{h}_{T_{j}}\bigg\|_1
\geq\|\bm{A}\bm{h}_{T_{01}}\|_1-\sum_{j\geq 2}^J\|\bm{A}\bm{h}_{T_{j}}\|_1\nonumber\\
\geq& (1-\delta_{t+s}^{lb})\|\bm{h}_{T_{01}}\|_{2}-(1+\delta_{t}^{ub})\sum_{j\geq 2}^J\|\bm{h}_{T_j}\|_{2},
\end{align}
where the last inequality is due to  $T_{01}=T_0\cup T_1$,  $|T_{0}|\leq s$,  $|T_{i}|\leq t$ for $i=1, 2, \cdots, J $,
  and $\bm{A}$ satisfies the  $(\ell_2,\ell_1)$-RIP condition of $t+s$ order.
% $\|h_{T_{01}}\|_0\leq t+s$ and $\|h_{T_j}\|_0\leq t$. By the $t+s$ order $(\ell_2,\ell_1)$-RIP condition, we have
%\begin{align}\label{e2.4}
%\|Ah\|_1&\geq\|A(h_{T_{01}})\|_1-\sum_{j\geq 2}\|A(h_{T_{j}})\|_1\nonumber\\
%&\geq (1-\delta_{t+s}^{lb})\|h_{T_{01}}\|_{2}-(1+\delta_{t}^{ub})\sum_{j\geq 2}\|h_{T_j}\|_{2}.
%\end{align}
Thus, to show \eqref{e:Ahlowerbound}, it  suffices to show that
\begin{align}\label{e:htjupperbound}
\sum_{j\geq 2}^J\|\bm{h}_{T_j}\|_2 \leq \frac{\|\bm{h}_{T_{01}}\|_2}{a(s,t;\alpha)}+\frac{2\|\bm{x}_{-\max(s)}\|_1}{\sqrt{t}-\alpha}.
\end{align}

Next, we move to prove \eqref{e:htjupperbound}.
%Therefore we need an upper bound for $\sum_{j\geq 2}\|h_{T_j}\|_{2}$.
%Note that
For $2\leq  j \leq J$, it follows from the definition of  $T_{j-1}$ that
\begin{align}\label{e:hiupperbound}
|h_i|\leq\min_{k\in T_{j-1}}|h_k|\leq\frac{\|\bm{h}_{T_{j-1}}\|_1-\alpha\|\bm{h}_{T_{j-1}}\|_2}{t-\alpha\sqrt{t}},
\end{align}
for any $i\in T_{j}$,
where  the last inequality is from  Lemma \ref{LocalEstimateL1-L2} and $|T_{j-1}|=t$ with  $2\leq  j \leq J$.
Then, for $2\leq  j \leq J$, one has
% which implies that
\begin{align*}
\|\bm{h}_{T_j}\|_2=\Big(\sum_{i\in T_j}|h_i|^2\Big)^{1/2}\leq \frac{\|\bm{h}_{T_{j-1}}\|_1-\alpha\|\bm{h}_{T_{j-1}}\|_2}{\sqrt{t}-\alpha},
\end{align*}
where the last inequality is from \eqref{e:hiupperbound} and $|T_{j}|\leq t$ with  $2\leq  j \leq J$.
Therefore, by the above inequality,
\begin{align*}
\sum_{j\geq 2}\|\bm{h}_{T_j}\|_2
\leq&\frac{\sum_{j\geq2}^J(\|\bm{h}_{T_{j-1}}\|_1-\alpha\|\bm{h}_{T_{j-1}}\|_2)}{\sqrt{t}-\alpha}\\
%\overset{(1)}{\leq}&\frac{\|h_{T_{0}^c}\|_1-\alpha\Big(\sum_{j\geq2}\|h_{T_{j-1}}\|_2^2\Big)^{1/2}}{\sqrt{t}-\alpha}\nonumber\\
\overset{(1)}{\leq}&\frac{\|\bm{h}_{T_{0}^c}\|_1-\alpha\|\bm{h}_{T_0^c}\|_2}{\sqrt{t}-\alpha}\\
\overset{(2)}{\leq}&\frac{\|\bm{h}_{T_{0}}\|_1+2\|x_{-\max(s)}\|_1+\alpha\|\bm{h}_{T_0}\|_2}{\sqrt{t}-\alpha}\\
%\leq&\frac{(\sqrt{s}+\alpha)\|h_{T_0}\|_2}{\sqrt{t}-\alpha}+\frac{2\|x_{-\max(s)}\|_1}{\sqrt{t}-\alpha}
\overset{(3)}{\leq}&\frac{\|\bm{h}_{T_{01}}\|_2}{a(s,t,;\alpha)}+\frac{2\|x_{-\max(s)}\|_1}{\sqrt{t}-\alpha}
\end{align*}
where
(1) is due to $T_0^c=\cup_{j=1}^J T_j$ and the fact that
$\sum_{j\geq2}^J \|\bm{h}_{T_{j-1}}\|_2\geq \Big(\sum_{j\geq2}^J\|h_{T_{j-1}}\|_2^2\Big)^{1/2}$,
(2) follows from \eqref{e:h-maxsupperbound2} and $T_0=\mathrm{supp}(\bm{h}_{\max(s)})$
and (3) is from $\|\bm{h}_{T_{0}}\|_1\leq \sqrt{|T_{0}|}\|\bm{h}_{T_{0}}\|_2$, $|T_0|\leq s$,
$T_{01}=T_0\cup T_1$
and  $a(s,t;k)=\frac{\sqrt{t}-\alpha}{\sqrt{s}+\alpha}$.
The proof is complete.

%where the second line follows from Lemma \ref{ConeconstraintinequalityforL1-L2}.

%Last, inserting (\ref{e2.5}) into (\ref{e2.4}), we have
%\begin{align}\label{e2.6}
%\|Ah\|_1&\geq (1-\delta_{t+s}^{lb})\|h_{T_{01}}\|_{2}-(1+\delta_{t}^{ub})\frac{\|h_{T_0}\|_2}{a(s,t)}
%-\frac{2\|x_{-\max(s)}\|_1}{\sqrt{t}-\alpha}\nonumber\\
%&\geq\Big(1-\delta_{t+s}^{lb}-\frac{(1+\delta_{t}^{ub})}{a(s,t;\alpha)}\Big)\|h_{T_{01}}\|_2
%-\frac{2\|x_{-\max(s)}\|_1}{\sqrt{t}-\alpha},
%\end{align}
%which gives a lower bound of $\|Ah\|_1$.
\end{proof}

\subsection{Exact Recovery under $(\ell_2,\ell_1)$-RIP\label{s2.2}}
\hskip\parindent

Now, we present our result for  the exact recovery
of $\bm{x}$ from \eqref{systemequationsnoise} with $\bm{z}=\bm{0}$ via (\ref{VectorL1-alphaL2-Exact}).

\begin{theorem}\label{ExactRecoveryviaVectorL1-alphaL2-Exact}
 For $0 <\alpha\leq 1$,
let $s\in[[1,n]]$, $k>0$ such that $ks\in\mathbb{Z}_+$ and  $a(s,ks;\alpha)=\frac{\sqrt{ks}-\alpha}{\sqrt{s}+\alpha}>1$. Let  $\bm{b}=\bm{Ax}$ and $\bm{x}$ be $s$-sparse.
If the measurement matrix $\bm{A}$ satisfies $(\ell_2,\ell_1)$-RIP with
\begin{align}\label{e2.0}
\delta_{ks}^{ub}+a(s,ks;\alpha)\delta_{(k+1)s}^{lb}<a(s,ks;\alpha)-1,
\end{align}
then \eqref{VectorL1-alphaL2-Exact} has unique $s$-sparse solution.
\end{theorem}

\begin{remark}\label{Simplercondition-ExactRecoveryviaVectorL1-alphaL2-Exact}
If % we take
$k\geq 4\alpha^2/(\sqrt{s}-\alpha)^2$, then the sufficient condition \eqref{e2.0} can be replaced by
\begin{align}\label{e2.9g}
\delta_{ks}^{ub}+\big(\sqrt{k}/2\big)\delta_{(k+1)s}^{lb}<\sqrt{k}/2-1.
\end{align}
In fact, by $k\geq 4\alpha^2/(\sqrt{s}-\alpha)^2$,
then %we have
$a(s,ks;\alpha)=\frac{\sqrt{ks}-\alpha}{\sqrt{s}+\alpha}\geq\sqrt{k}/2>1$,
Furthermore, by \eqref{e2.9g},
 \begin{align*}
%\rho_1=
1-\delta_{(k+1)s}^{lb}-\frac{1+\delta_{ks}^{ub}}{a(s,ks;\alpha)}
\geq  1-\delta_{(k+1)s}^{lb}-\frac{2(1+\delta_{ks}^{ub})}{\sqrt{k}}>0,
\end{align*}
which implies \eqref{e2.0}.
 %Then we can replace (\ref{e2.5}) with
%\begin{align*}
%\sum_{j\geq 2}\|h_{T_j}\|_2&
%\leq\frac{\|h_{T_0}\|_2}{a},
%\end{align*}
%and replace $\rho_1$ by
%$$
%\tilde{\rho}_1=1-\delta_{(k+1)s}^{lb}-\frac{(1+\delta_{ks}^{ub})}{a}.
%$$
%Then we can get an simpler condition
%\begin{align}\label{e2.9}
%\delta_{ks}^{ub}+\sqrt{k}/2\delta_{(k+1)s}^{lb}<\sqrt{k}/2-1.
%\end{align}
\end{remark}

%%%%%%%%%%%%%%%%%%%%%%%%%%%%%%%%%%%%%%%%%%%%%%%%%%%%%%%%%%%%%%%%%%%%
%%%%%%%%%%%%%%%%%%%%%%   section 3 %%%%%%%%%%%%%%%%%%%%%%%%%%%%%%%%%
%%%%%%%%%%%%%%%%%%%%%%%%%%%%%%%%%%%%%%%%%%%%%%%%%%%%%%%%%%%%%%%%%%%
\section{Stable Recovery via $\ell_{1}-\alpha\ell_{2}$ Minimization\label{s3}}
\hskip\parindent

In the bounded noisy case, we will consider the stable recovery of the signal $\bm{x}$ from \eqref{systemequationsnoise}
via models (\ref{VectorL1-alphaL2-LAD}) and (\ref{VectorL1-alphaL2-DS}).

\subsection{Stable Recovery Under $(\ell_2, \ell_1)$-RIP\label{s3.1}}
\hskip\parindent
In the $\ell_1$ bounded noisy case, we obtain the sufficient conditions
for the stable  recovery of the signal $\bm{x}$ from \eqref{systemequationsnoise} via the $\ell_{1}-\alpha \ell_{2}$  minimization model (\ref{VectorL1-alphaL2-LAD}) in the following theorem.
\begin{theorem}\label{StableRecoveryviaVectorL1-alphaL2-LAD}
Consider  $\bm{b}=\bm{Ax}+\bm{z}$ with $\|\bm{z}\|_1\leq \eta_1$.
For some $s\in[[1,n]]$ and $0< \alpha\leq 1$, let  $k>0$ such that $ks\in\mathbb{Z}_+$ and  $a(s,ks; \alpha)=\frac{\sqrt{ks}-\alpha}{\sqrt{s}+\alpha}>1$.
Let  $\hat{\bm{x}}^{\ell_1}$ be the minimizer of  (\ref{VectorL1-alphaL2-LAD}).
If the measurement matrix $\bm{A}$ satisfies (\ref{e2.0}), then
\begin{align}\label{e:stablel1upper}
&\|\hat{\bm{x}}^{\ell_1}-\bm{x}\|_2
\leq\frac{2(2\sqrt{k}+1)\sqrt{s}}{(2\sqrt{2}\sqrt{s}-\alpha)\rho_{ks}}\eta_1
\nonumber\\
&\ \ +\frac{\sqrt{s}}{2\sqrt{ks}-\alpha}\Big(\frac{(2\sqrt{k}+1)(1+\delta_{ks}^{ub})\sqrt{s}}
{\rho_{ks}(\sqrt{ks}-\alpha)}+1\Big)
\frac{2\|\bm{x}_{-\max(s)}\|_1}{\sqrt{s}},
\end{align}
where $\rho_{ks}=1-\delta_{(k+1)s}^{lb}-\frac{(1+\delta_{ks}^{ub})}{a(s,ks;\alpha)}$.
\end{theorem}

\begin{remark}\label{Simplercondition-StableRecoveryviaVectorL1-alphaL2-LAD}
Similar to  the discussion in Remark \ref{Simplercondition-ExactRecoveryviaVectorL1-alphaL2-Exact}, when
\begin{align*}
\delta_{ks}^{ub}+(\sqrt{k}/2)\delta_{(k+1)s}^{lb}<\sqrt{k}/2-1,
\end{align*}
the solution $\hat{\bm{x}}^{\ell_1}$ of (\ref{VectorL1-alphaL2-LAD}) satisfies
\begin{align*}
&\|\hat{\bm{x}}^{\ell_1}-\bm{x}\|_2
\leq\frac{2(2\sqrt{k}+1)\sqrt{s}}{(2\sqrt{ks}-\alpha)\tilde{\rho}_{ks}}\eta_1\\
&+\frac{\sqrt{s}}{2\sqrt{ks}-\alpha}\Big(\frac{(2\sqrt{k}+1)(1+\delta_{ks}^{ub})
\sqrt{s}}{\tilde{\rho}_{ks}(\sqrt{ks}-\alpha)}+1\Big)
\frac{2\|\bm{x}_{-\max(s)}\|_1}{\sqrt{s}},
\end{align*}
where $\tilde{\rho}_{ks}=1-\delta_{(k+1)s}^{lb}-\frac{(1+\delta_{ks}^{ub})}{\sqrt{k}/2}$.
\end{remark}

Now, we consider the recovery model
\eqref{systemequationsnoise} with
$\|\bm{A}^{*}\bm{z}\|_\infty\leq \eta_2$.
%Next, we consider the stable recovery under Dantzig selector constraint.
\begin{theorem}\label{StableRecoveryviaVectorL1-alphaL2-DS}
Consider $\bm{b}=\bm{Ax}+\bm{z}$ with $\|\bm{A}^{*}\bm{z}\|_\infty\leq \eta_2$. For some $s\in[[1,n]]$
and $0<\alpha\leq 1$, let
%Let $s\in[[1,n]]$,
$k>0$ such that $ks\in\mathbb{Z}_+$,  $a(s,ks;\alpha)=(\sqrt{ks}-\alpha)/(\sqrt{s}+\alpha)>2$ and $b(s, k; \alpha)=8(2\sqrt{ks}-\alpha)/(17\alpha(2\sqrt{k}+1))>1$ satisfying   $a(s,ks;\alpha)b(s, k; \alpha)<a(s,ks;\alpha)+b(s, k; \alpha)$. Let $\hat{x}^{DS}$ be the minimizer
of the  $\ell_{1}-\alpha \ell_{2}$ minimization model (\ref{VectorL1-alphaL2-DS}). If the  measurement matrix $\bm{A}$ satisfies the $(\ell_2,\ell_1)$-RIP condition with
\begin{align}\label{RIPCondition3}
&\big(b(s,k;\alpha)+1\big)\delta_{ks}^{ub}+a(s,ks;\alpha)b(s,k;\alpha)\delta_{(k+1)s}^{lb}\nonumber\\
&<a(s,ks;\alpha)b(s,k;\alpha)-b(s,k;\alpha)-1,
\end{align}
then
\begin{align*}
\|\hat{\bm{x}}^{DS}&-\bm{x}\|_2\\
&\leq\frac{\sqrt{s}\varrho}{\sqrt{s}-\alpha\varrho}\frac{2\|\textbf{x}_{-\max(s)}\|_1}{\sqrt{s}}\\
&\hspace*{12pt}+\frac{2(2\sqrt{k}+1)\big((1+\delta_{ks}^{ub})+a(s;\alpha,k)\rho_{ks}\big)ms}
{\sqrt{k}(\sqrt{s}-\alpha\varrho)(1+\delta_{ks}^{up})\rho_{ks}^2}\eta_2
\end{align*}
where $\varrho=\frac{1}{2\sqrt{k}}\bigg(\frac{17(2\sqrt{k}+1)(1+\delta_{ks}^{ub})}{16a(s,ks;\alpha)\rho_{ks}}+1\bigg)$.
\end{theorem}

\begin{remark}\label{Simplercondition-StableRecoveryviaVectorL1-alphaL2-DS}
The conditions in Theorem \ref{StableRecoveryviaVectorL1-alphaL2-DS} seem strict. In fact,
these  conditions can be satisfied. For example, for $\alpha=1$, if we take $k=16$, then
$$
a(s,ks;\alpha)=\frac{4\sqrt{s}-1}{\sqrt{s}+1}=:a(s), b(s,k;\alpha)=\frac{8(8\sqrt{s}-1)}{153}=:b(s).
$$
If we restrict $7\leq s\leq 14$, we can check that $a(s)>2$, $b(s)>1$ and $a(s)b(s)<a(s)+b(s)$. Therefore, $(\ell_2,\ell_1)$-RIP condition (\ref{RIPCondition3}) can be formulated as
$$
\big(b(s)+1\big)\delta_{ks}^{ub}+a(s)b(s)\delta_{(k+1)s}^{lb}<a(s)b(s)-b(s)-1.
$$
And if we take $\delta_s^{lb}=\delta_s^{ub}=\delta_s$ in Remark \ref{lqRIP-Remark}, then condition (\ref{RIPCondition3}) can be simplified as
$$
\delta_{17s}<\frac{192s-305\sqrt{s}-137}{320s+113\sqrt{s}+153}.
$$
\end{remark}

%%%%%%%%%%%%%%%%%%%%%%%%%%%%%%%%%%%%%%%%%%%%%%%%%%%%%%%%%%%%%%%%%%%%
%%%%%%%%%%%%%%%%%%%%%%   section 4 %%%%%%%%%%%%%%%%%%%%%%%%%%%%%%%%%
%%%%%%%%%%%%%%%%%%%%%%%%%%%%%%%%%%%%%%%%%%%%%%%%%%%%%%%%%%%%%%%%%%%
\section{Computational Approach for $\ell_{1}-\alpha\ell_{2}$-PLAD \label{s4}}
\hskip\parindent

In this section, we consider how to  solve the unconstraint $\ell_{1}-\alpha \ell_{2}$-PLAD problem (\ref{VectorL1-alphaL2-PLAD}).
%compute the $\ell_{1}-\alpha \ell_{2}$ minimization. We solve unconstraint $\ell_{1}-\alpha \ell_{2}$-minimization model (\ref{VectorL1-alphaL2-PLAD}). We design an algorithm by  ADMM \cite{BPCBJ2011}.
First, by splitting  the term $\|\bm{A}\bm{x}-\bm{b}\|_1$, we get an equivalent problem of (\ref{VectorL1-alphaL2-PLAD}) as follows
\begin{align}\label{VectorL1-2-LAD-Equivalent}
\min_{\bm{x}\in\mathbb{R}^n, \bm{y}\in\mathbb{R}^m}\lambda(\|\bm{x}\|_1-\alpha\|\bm{x}\|_2)+\|\bm{y}\|_1
~\text{subject~to~}\bm{A}\bm{x}-\bm{y}=\bm{b}.
\end{align}
Let
\begin{align}\label{ALagrange-LAD}
\mathcal{L}_{\gamma}(\bm{x},\bm{y};\bm{w})
=&\lambda(\|\bm{x}\|_1-\alpha\|\bm{x}\|_2)+\|\bm{y}\|_1
 \nonumber\\
&-\langle \bm{w}, \bm{A}\bm{x}-\bm{y}-\bm{b}\rangle+\frac{\gamma}{2}\|\bm{A}\bm{x}-\bm{y}-\bm{b}\|_2^2\nonumber\\
=&\lambda(\|\bm{x}\|_1-\alpha\|\bm{x}\|_2)+\|\bm{y}\|_1\nonumber\\
&+\frac{\gamma}{2}\Big\|\bm{A}\bm{x}-\bm{y}-\bm{b}-\frac{\bm{w}}{\gamma}\Big\|_2^2
-\frac{1}{2\gamma}\|\bm{w}\|_2^2,
\end{align}
which is  the augmented Lagrangian function of (\ref{VectorL1-2-LAD-Equivalent}) with the Lagrangian multiplier $\bm{w}\in\mathbb{R}^m$ and a penalty parameter $\gamma>0$.
Given $(\bm{x},\bm{y};\bm{w})\in\mathbb{R}^n\times\mathbb{R}^m\times\mathbb{R}^m$,  %the iterative scheme
iterations for (\ref{ALagrange-LAD}) are
\begin{align}\label{ADMM-LAD}
\begin{cases}
\bm{x}^{k+1}=\arg\min_{\bm{x}\in\mathbb{R}^n}\mathcal{L}_{\gamma}(\bm{x},\bm{y}^k;\bm{w}^k),\\
\bm{y}^{k+1}=\arg\min_{\bm{y}\in\mathbb{R}^m}\mathcal{L}_{\gamma}(\bm{x}^{k+1},\bm{y};\bm{w}^{k}),\\
\bm{w}^{k+1}=\bm{w}^k-\gamma(\bm{A}\bm{x}^{k+1}-\bm{y}^{k+1}-\bm{b}).\\
\end{cases}
\end{align}

Now, we move to  consider (\ref{ADMM-LAD}).
By \eqref{ALagrange-LAD}, the  $\bm{x}$-related subproblem in \eqref{ADMM-LAD} is
 equivalent to
\begin{align}\label{x-subproblem-Equivalent}
\bm{x}^{k+1}&=\arg\min_{\bm{x}\in\mathbb{R}^n}
\frac{\gamma}{2}\bigg\|\bm{A}\bm{x}-\bigg(\bm{b}+\bm{y}^k+\frac{\bm{w}^k}{\gamma}\bigg)\bigg\|_2^2
+\lambda\|\bm{x}\|_1\nonumber\\
&\hspace*{12pt}-\lambda\alpha\|\bm{x}\|_2\nonumber\\
&=\arg\min_{\bm{x}\in\mathbb{R}^n}\bigg(\frac{\gamma}{2}\|\bm{A}\bm{x}-\bar{\bm{b}}^k\|_2^2
+\lambda\|\bm{x}\|_1\bigg)-\lambda\alpha\|\bm{x}\|_2\nonumber\\
&=\arg\min_{\bm{x}\in\mathbb{R}^n}\mathcal{E}(\bm{x})-\mathcal{F}(\bm{x}),
\end{align}
where the second equality is from $\bar{\bm{b}}^k=\bm{b}+\bm{y}^k+\bm{w}^k/\gamma$,
and the last equality is due to $\mathcal{E}(\bm{x})=\frac{\gamma}{2}\|\bm{A}\bm{x}-\bar{\bm{b}}^k\|_2^2
+\lambda\|\bm{x}\|_1$ and $\mathcal{F}(\bm{x})=\lambda\alpha\|\bm{x}\|_2$.
In terms of the analysis for  \cite[(3.1)]{YLHX2015},
we  solve $\bm{x}$-related subproblem (\ref{x-subproblem-Equivalent})  using the DCA.
To implement the DCA, one iteratively computes
\begin{align*}
\bm{x}^{k+1}&=\arg\min_{\bm{x}\in\mathbb{R}^n}\mathcal{E}(\bm{x})-(\mathcal{F}(\bm{x}^k)+\langle \bm{h}^k, \bm{x}-\bm{x}^k \rangle),
\end{align*}
where $\bm{h}^k\in\partial \mathcal{F}(\bm{x}^k)$. Note that $\mathcal{F}(\bm{x})$ is differentiable with the gradient
\begin{align*}
\partial \mathcal{F}(\bm{x})
\begin{cases}
=\lambda\alpha\frac{\bm{x}}{\|\bm{x}\|_2},&\text{for~all~}\bm{x}\neq \bm{0};\\
\ni \bm{0}, &\bm{x}=\bm{0}.
\end{cases}
\end{align*}
Therefore, if $\bm{x}^k=\bm{0}$,
$$
\bm{x}^{k+1}=
\arg\min_{\bm{x}\in\mathbb{R}^n}\frac{\gamma}{2}\|\bm{A}\bm{x}-\bar{\bm{b}}^k\|_2^2
+\lambda\|\bm{x}\|_1,
$$
otherwise,
$$
\bm{x}^{k+1}=\arg\min_{\bm{x}\in\mathbb{R}^n}\frac{\gamma}{2}\|\bm{A}\bm{x}-\bar{\bm{b}}^k\|_2^2
+\lambda\|\bm{x}\|_1-\lambda\alpha\langle \bm{x}, \frac{\bm{x}}{\|\bm{x}\|_2}\rangle.
$$
Thus the strategy to iterate is as follows:
\begin{align}\label{DCA-LAD-1}
\bm{x}^{k+1}
&=\arg\min_{\bm{x}\in\mathbb{R}^n}\frac{\gamma}{2}\|\bm{A}\bm{x}-\bar{\bm{b}}^k\|_2^2
+\lambda\|\bm{x}\|_1-\lambda\alpha\langle \bm{x}, \bm{v}^k\rangle\nonumber\\
&=:\arg\min_{\bm{x}\in\mathbb{R}^n}\mathcal{G}_{\gamma}(\bm{x},\bar{\bm{b}}^k),
\end{align}
where
\begin{align}\label{Subgradient-L2}
\bm{v}^k=
\begin{cases}
\frac{\bm{x}^k}{\|\bm{x}^k\|_2},&\text{if~}\bm{x}^k\neq \bm{0};\\
\bm{0}, &\text{if}~\bm{x}^k=\bm{0}.
\end{cases}
\end{align}

By taking subdifferential of $\mathcal{G}_{\gamma}(\bm{x},\bar{\bm{b}}^k)$ at $\bm{x}=\bm{x}^{k+1}$, we have
$$
\bm{0}=\gamma \bm{A}^*\bm{A}\bm{x}^{k+1}+\lambda\partial\|\bm{x}^{k+1}\|_1
-\bigg(\gamma \bm{A}^*\bar{\bm{b}}^k+\lambda\alpha\bm{v}^{k}\bigg).
$$
Whenever $\bm{A}^*\bm{A}=c\bm{I}_{n}$, which essentially implies that the columns of the design matrix $\bm{A}$ are orthogonal, the closed-form solution of (\ref{DCA-LAD-1}) is given by the soft shrinkage operator. However, the assumption $m\leq n$ indicates that the rank of $\bm{A}$ is no bigger than $m$ and thus the rank of $\bm{A}^*\bm{A}$ should be much smaller than $n$. Therefore, $\bm{A}^*\bm{A}$ is not the identity matrix in $\mathbb{R}^{n\times n}$ when $m\leq n$, and the closed-form solution of (\ref{DCA-LAD-1}) is not available for this case.

To alleviate  the above difficulty,
we adopt the strategy of linearizing the quadratic term, which comes from Wang and Yuan \cite{WY2012}.  In fact, the quadratic term $\frac{\gamma}{2}\|\bm{A}\bm{x}-\bar{\bm{b}}^k\|_2^2$ can be linearized:
\begin{align*}
&\frac{\gamma}{2}\|\bm{A}\bm{x}-\bar{\bm{b}}^k\|_2^2\nonumber\\
&\approx
\frac{\gamma}{2}\bigg(\|\bm{A}\bm{x}^k-\bar{\bm{b}}^k\|_2^2
+\Big\langle 2\bm{A}^*(\bm{A}\bm{x}^k-\bar{\bm{b}}^k),\bm{x}-\bm{x}^k\Big\rangle\nonumber\\
&\hspace*{12pt}+\frac{1}{\mu}\|\bm{x}-\bm{x}^k\|_2^2\bigg)\\
&=\frac{\gamma}{2}\bigg(\|\bm{A}\bm{x}^k-\bar{\bm{b}}^k\|_2^2
+\frac{1}{\mu}\|\bm{x}-\bm{x}^k+\mu \bm{A}^*(\bm{A}\bm{x}^k-\bar{\bm{b}}^k)\|_2^2\nonumber\\
&\hspace*{12pt}-\mu\|\bm{A}^*(\bm{A}\bm{x}^k-\bar{\bm{b}}^k)|_2^2\bigg).
\end{align*}
Then we can approximate subproblem of (\ref{DCA-LAD-1}) by
\begin{align}\label{DCA-LAD-2}
\bm{x}^{k+1}
&=\arg\min_{\bm{x}\in\mathbb{R}^n}
\frac{\gamma}{2\mu}\|\bm{x}-\bm{x}^k+\mu \bm{A}^*(\bm{A}\bm{x}^k-\bar{\bm{b}}^k)\|_2^2\nonumber\\
&\hspace*{12pt}+\lambda\|\bm{x}\|_1-\lambda\alpha\langle \bm{x}, \bm{v}^k\rangle\nonumber\\
&=:\arg\min_{x\in\mathbb{R}^n}\mathcal{H}_{\gamma,\lambda}(\bm{x},\bar{\bm{b}}^k).
\end{align}

By taking subdifferential of $\mathcal{H}_{\gamma,\lambda}(\bm{x},\bar{\bm{b}}^k)$ at $\bm{x}=\bm{x}^{k+1}$, we have
\begin{align*}
\bm{0}&=\frac{\gamma}{\mu}\Bigg(\bm{x}^{k+1}+\frac{\lambda\mu}{\gamma}\partial\|\bm{x}^{k+1}\|_1\nonumber\\
&\hspace*{12pt}-\bigg(\bm{x}^k-\mu \bm{A}^*(\bm{A}\bm{x}^k-\bar{\bm{b}}^k)+\frac{\lambda\alpha\mu}{\gamma}\bm{v}^{k+1}\bigg)\Bigg).
\end{align*}
Therefore,
\begin{align}\label{x-ADMM-LAD}
&\bm{x}^{k+1}\nonumber\\
&=S\Bigg(\bm{x}^k-\mu \bm{A}^*\bigg(\bm{A}\bm{x}^k-\bm{b}-\bm{y}^k-\frac{\bm{w}^k}{\gamma}\bigg)
+\frac{\lambda\alpha\mu}{\gamma}\bm{v}^{k+1},
\frac{\lambda\mu}{\gamma}\Bigg)
\end{align}
where
\begin{align*}
(S(\bm{x}, r))_i=\text{sign}(x_i)\max\{|x_i|-r,0\}
\end{align*}
is the soft thresholding operator.

Next, we turn our attention to deal with $\bm{y}$-related subproblem in \eqref{ADMM-LAD}. The $\bm{y}$-related subproblem is just a constrained least squares
problem
\begin{align*}
\bm{y}^{k+1}&=\arg\min_{\bm{y}\in\mathbb{R}^m}\mathcal{L}_{\gamma}(\bm{x}^{k+1},\bm{y};\bm{w}^{k})\\
&=\arg\min_{\bm{y}\in\mathbb{R}^m}\|\bm{y}\|_1
+\frac{\gamma}{2}\bigg\|\bm{y}-\Big(\bm{A}\bm{x}^{k+1}-\bm{b}-\frac{\bm{w}^{k}}{\gamma}\Big)\bigg\|_2^2,
\end{align*}
which implies that
\begin{align}\label{y-ADMM-LAD}
\bm{y}^{k+1}=S\Bigg(\bm{A}\bm{x}^{k+1}-\bm{b}-\frac{\bm{w}^{k+1}}{\gamma},
\frac{1}{\gamma}\Bigg).
\end{align}

Now, we  present the algorithm applying the linearized ADMM and DCA to
solve the unconstrained $\ell_{1}-\alpha \ell_{2}$-PLAD  problem (\ref{VectorL1-alphaL2-PLAD}).
\begin{algorithm}\label{al:LADMg}
\centering
\caption{$\ell_{1}-\alpha\ell_{2}$LA for solving (\ref{VectorL1-alphaL2-PLAD})}
\vspace{-0mm}
\begin{tabular}{@{}ll}
%\hline
\textbf{~~Input}~~$\bm{A}$, $\bm{b}$, $\alpha$, $\lambda$, $\gamma$,
$\mu$,  $(\bm{x}^0,\bm{y}^0;\bm{w}^0)$, $k=1$ .\\
 %model parameter
 %$0<\alpha\leq 1$, penalty parameter $\lambda>0$ and regularized parameters $\gamma>0, 0<\rho_1<1/{\|\bm{A}^*\bm{A}\|_{2\rightarrow 2}}, 0<\rho_2<1$. Given a sequence of nonnegative parameters $\{\alpha_k\}_{k=0}^{\infty}$ and $(\bm{x}^0,\bm{y}^0;\bm{w}^0)$. Let $(\bm{x}^{-1},\bm{y}^{-1};\bm{w}^{-1})=(\bm{x}^{0},\bm{y}^{0};\bm{w}^{0})$ and $k=0$.~ \\

~~\textbf{While}~~some stopping criterion is not satisfied \textbf{do}\\
~~~~\textbf{1.}\ \ \ \ Compute $\bm{v}^{k}$ by (\ref{Subgradient-L2}).\\
~~~~\textbf{2.}\ \ \ \ Update $\bm{x}^{k+1}$ by (\ref{x-ADMM-LAD}).\\
~~~~\textbf{3.}\ \ \ \ Update $\bm{y}^{k+1}$ by (\ref{y-ADMM-LAD}).\\
~~~~\textbf{4.}\ \ \ \ $\bm{w}^{k+1}=\bm{w}^k-\gamma(\bm{A}\bm{x}^{k+1}-\bm{b}-\bm{y}^{k})$.\\
~~~~\textbf{5.}\ \ \ \ $k=k+1$.\\

\textbf{~~End}
\vspace{-3.5pt} \\
%\hline
\end{tabular}
    \end{algorithm}
%\noindent\rule[0.25\baselineskip]{\textwidth}{1pt}

\begin{remark}
In Algorithm 1,
$\alpha$ is a model parameter  and satisfies  $0<\alpha\leq 1$, $\lambda>0$ is a
 penalty parameter,  $\gamma>0, 0<\mu<1/{\|\bm{A}^{*}\bm{A}\|_{2\rightarrow 2}}$ are regularized parameters.
\end{remark}

\section{Numerical Experiments of $\ell_{1}-\alpha\ell_{2}$-PLAD\label{s5}}
\hskip\parindent

In this section, we will present numerical experiments for sparse signals and compressible images to demonstrate the efficiency of $\ell_{1}-\alpha\ell_{2}$LA algorithm.

\subsection{Sparse Signal Recovery\label{s5.1}}
\hskip\parindent

In this subsection, we apply the proposed $\ell_1-\alpha \ell_2$LA algorithm to reconstruct sparse signals.
We also compare our $\ell_1-\alpha \ell_2$LA numerically with some efficient methods in the literature, including YALL1 \cite{YZ2011} for penalized LAD model
\begin{align}\label{VectorL1-PLAD}
\min_{\bm{x}\in\mathbb{R}^n}~\lambda\|\bm{x}\|_{1}+\|\bm{A}\bm{x}-\bm{b}\|_1,
\end{align}
and LqLA-ADMM \cite{WPYYL2017}
\begin{align}\label{VectorLq-PLAD}
\min_{\bm{x}\in\mathbb{R}^n}~\lambda\|\bm{x}\|_{q}^q+\|\bm{A}\bm{x}-\bm{b}\|_{1,\varepsilon}
\end{align}
with $\varepsilon>0$ is an approximation parameter, where $\|\bm{y}\|_{1,\varepsilon}=\sum_{j}(y_j^2+\varepsilon^2)^{1/2}$ and $0<q<1$.

We consider two types of impulsive noises \cite{W2013,WLLQY2016,WPYYL2017}.

(1) Gaussian Mixture Noise \cite{BG1986,TMM1985,SG1994}: we consider a typical two-term Gaussian mixture
model with probability density function (pdf) given by
$$
(1-\xi)\mathcal{N}(0,\sigma^2)+\xi\mathcal{N}(0,\kappa\sigma^2),
$$
where $0\leq\xi<1$ and $\kappa>1$, i.e., part of the noise variables $z_j$ are $\mathcal{N}(0,\sigma^2)$ random variables and part of them are $\mathcal{N}(0,\kappa\sigma^2)$ random variables. Here the two parameters $\xi$ and $\kappa>1$ respectively control the ratio and the strength of outliers in the noise. And the first term stands for the nominal background noise, e.g., Gaussian thermal noise, while the second term describes the impulsive behavior of the noise.

(2) Symmetric $\tau$-stable ($S\tau S$) Noise \cite{ST1994,N2012}: Except for a few known cases, the $S\tau S$ distributions do not have analytical formulations. The characteristic function of a zero-location
$S\tau S$ distribution can be expressed as
$$
\phi(\omega)=\exp(j\tau\omega-\gamma^{\tau}|\omega|^{\tau}),
$$
where $0<\tau\leq2$ is the characteristic exponent and $\gamma>0$ is the scale parameter or dispersion. The characteristic exponent measures the thickness of the tail of the distribution. The smaller the value of $\tau$, the heavier the tail of the distribution and the more impulsive the noise is. We can see that the $S\tau S$ distribution becomes the Gaussian distribution with variance $2\gamma^2$ when $\tau=2$, and it reduces to the Cauchy distribution when $\tau=1$.
The symmetric $1$-stable noise is heavy tail noise.

In our experiments, we test two classes measurement matrices with different coherence.
The coherence of a matrix $\bm{A}$ is the maximum absolute value of
the cross-correlations between the columns of $\bm{A}$, namely,
$$
\mu(\bm{A}):=\max_{i\neq j}\frac{|\langle \bm{A}_i,\bm{A}_j\rangle|}{\|\bm{A}_i\|_2\|\bm{A}_j\|_2}.
$$
This concept is introduced in \cite{DH2001}.

The first class: $\bm{A}$ is a random Gaussian matrix, %which is defined as
i.e.,
$$
\bm{A}_i\sim \mathcal{N}(0,I_m/m),~i=1,\ldots,n.
$$
which is incoherent and having small RIP constants with high probability.
%Thus,they  fit for sparse signal recovery.

The second class: $\bm{A}$ is a more ill-conditioned sensing matrix of significantly higher coherence.
Here, $\bm{A}$ is a randomly oversampled partial DCT matrix, which  is defined as
\begin{align*}
\bm{A}_i = \frac{1}{\sqrt{m}}\cos(2\pi\xi/F)
\end{align*}
where $\xi\in\mathbb{R}^m\sim\mathcal{U}([0,1]^m)$  the uniformly and independently distributed in $[0,1]^m$ , and $F\in\mathbb{N}$ is the refinement factor. Actually it is the real part of the random partial Fourier matrix analyzed in \cite{FL2011}. The number $F$ is closely related to the conditioning of $\bm{A}$ in the sense that $\mu(\bm{A})$ tends to get larger as $F$ increases. For $\bm{A}\in\mathbb{R}^{32\times 640}$, $\mu(\bm{A})$ easily exceeds 0.99 when $F =10$.  Although $\bm{A}$ sampled in this way does not have good RIP by any means, it is still possible to recover the sparse signal $\bm{x}$ provided its spikes are sufficiently separated.

\begin{figure*}[htbp!]
\centering
\begin{minipage}{5cm}
\centering
\includegraphics[width=5cm,height=5cm]{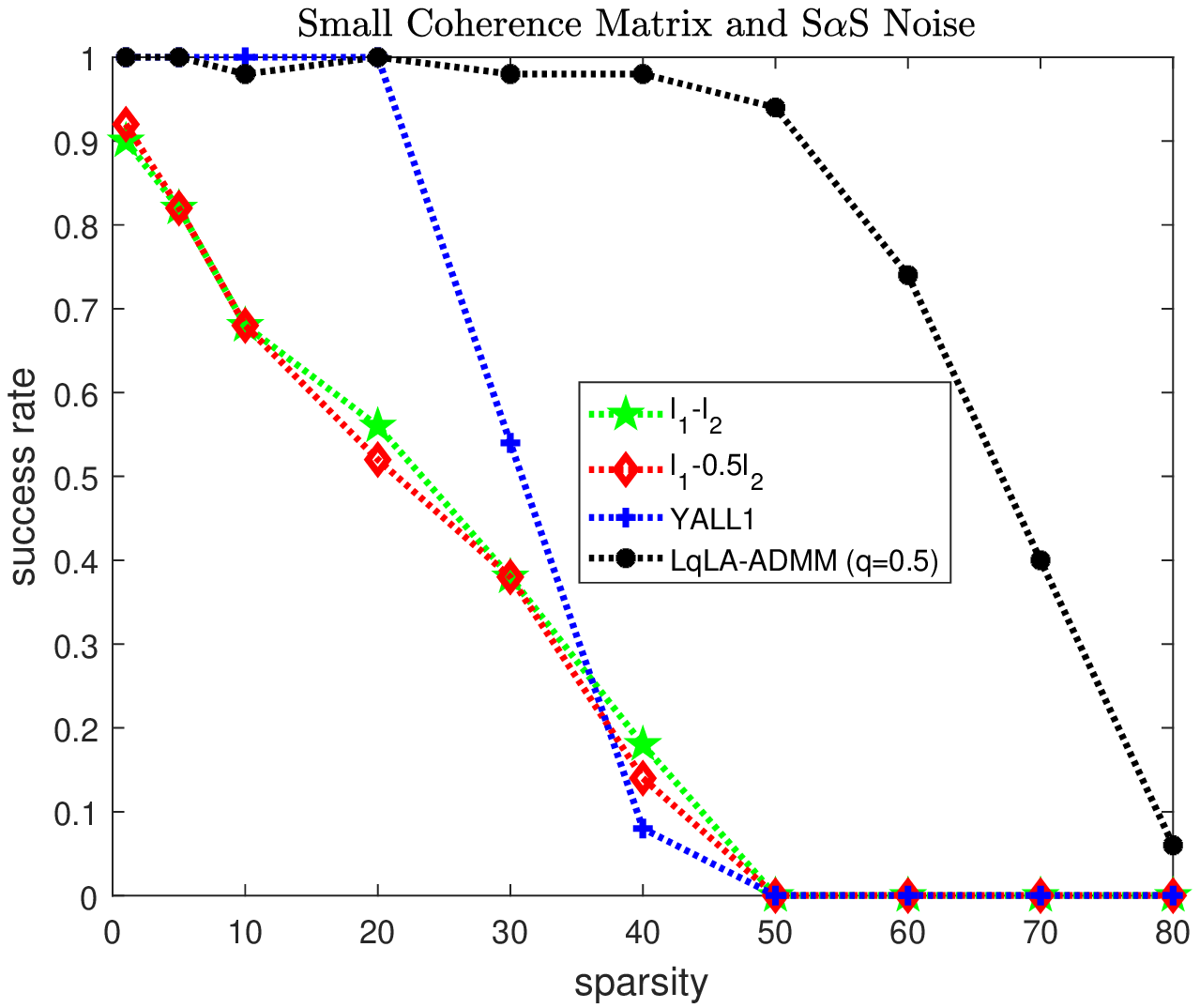}%\\[-0.3cm]
\end{minipage}
\centering
\begin{minipage}{5cm}
\centering
\includegraphics[width=5cm,height=5cm]{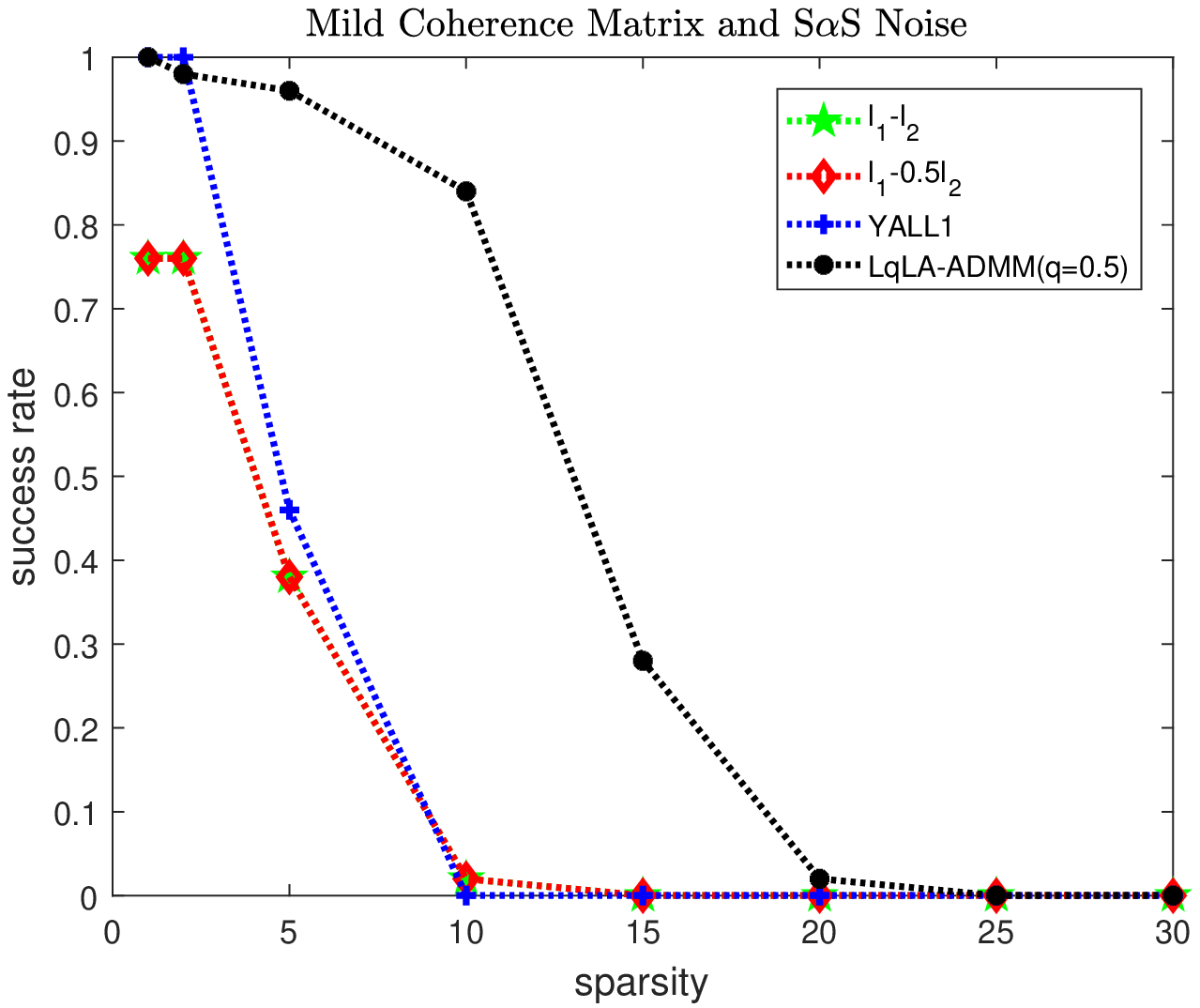}%\\[-0.3cm]
\end{minipage}
\centering
\begin{minipage}{4cm}
\centering
\includegraphics[width=5cm,height=5cm]{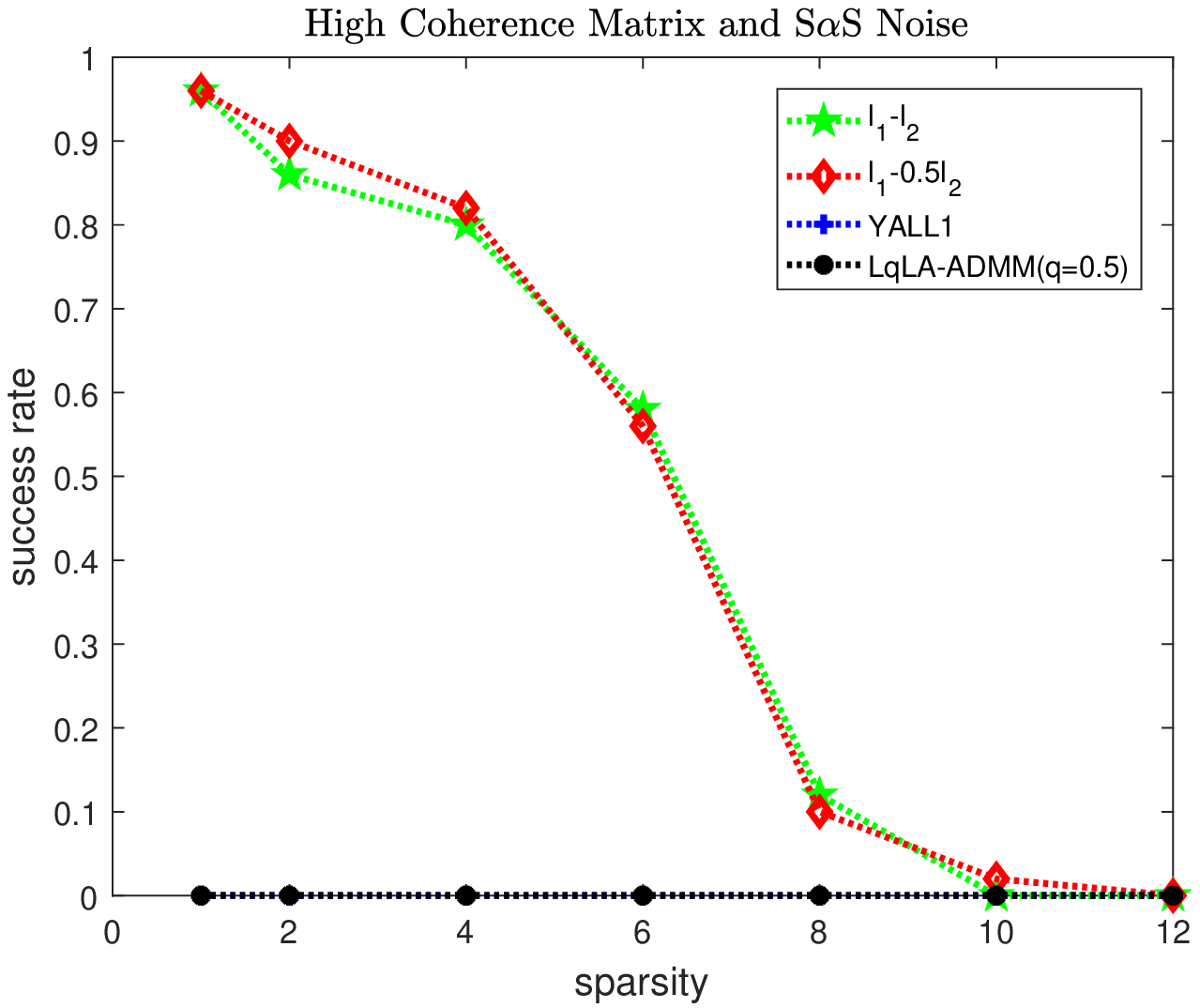}%\\[-0.3cm]
\end{minipage}
\\
\caption{\label{f4.1} $S\tau S$ noise. Left: $m=128, n=
256, s=1,5,10, 20,\ldots,80$, $\bm{A}\in\R^{m\times n}$ has small coherence with $\mu(\bm{A})<0.35$; Middle: $m=64, n=1024, s=1,2,5,10,15,\ldots,30$,    $\bm{A}\in\R^{m\times n}$ has mild coherence with $0.5<\mu(\bm{A})<0.6$;
Right:  $m=32, n=640, s=1,2,4,6,\ldots,12$, $\bm{A}\in\R^{m\times n}$ has  high coherence with $\mu(\bm{A})>0.99$.}	
\end{figure*}
\begin{figure*}[htbp!]
\centering
\begin{minipage}{5cm}
\centering
\includegraphics[width=5cm,height=5cm]{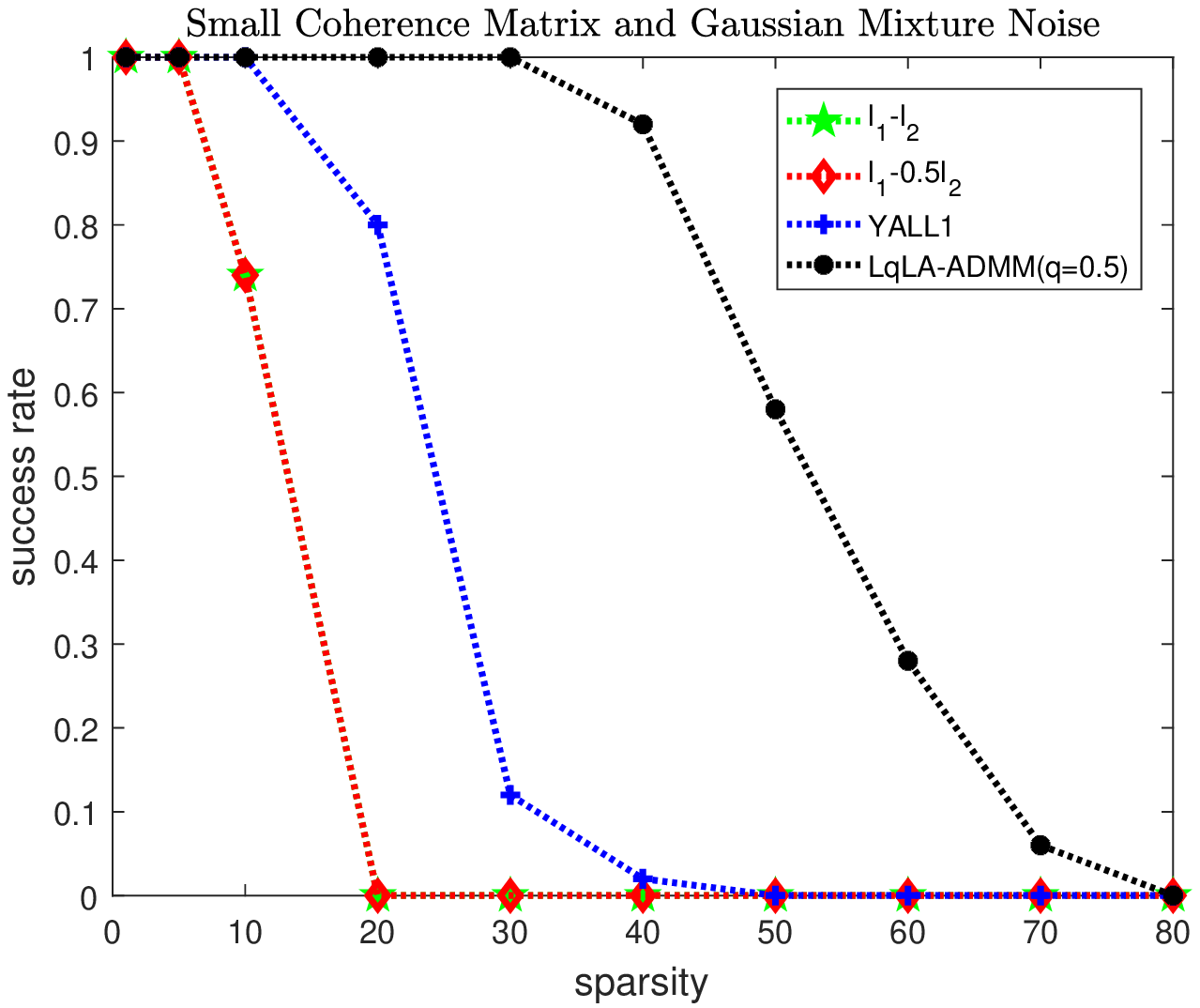}%\\[-0.3cm]
\end{minipage}
\centering
\begin{minipage}{5cm}
\centering
\includegraphics[width=5cm,height=5cm]{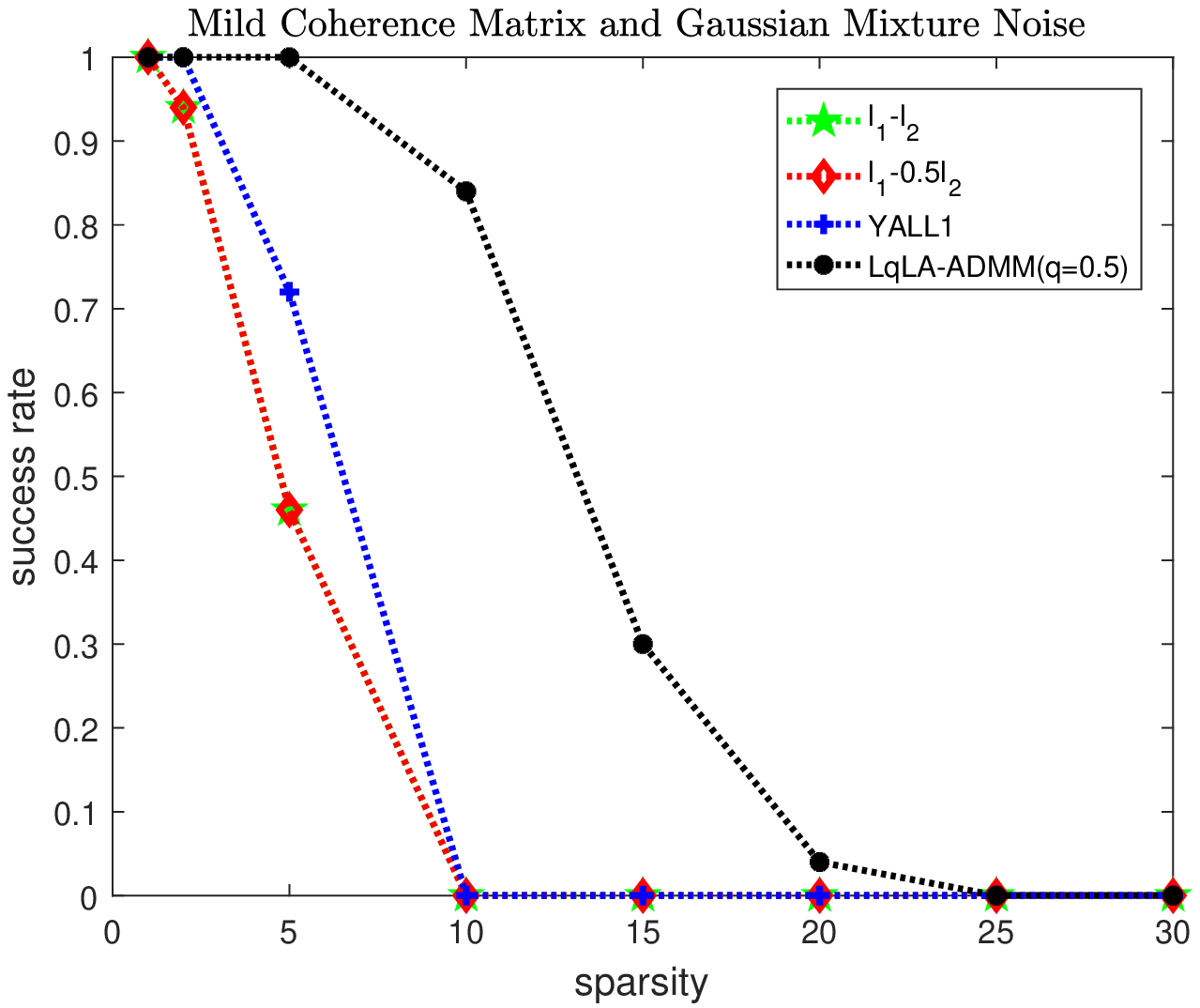}%\\[-0.3cm]
\end{minipage}
\centering
\begin{minipage}{4cm}
\centering
\includegraphics[width=5cm,height=5cm]{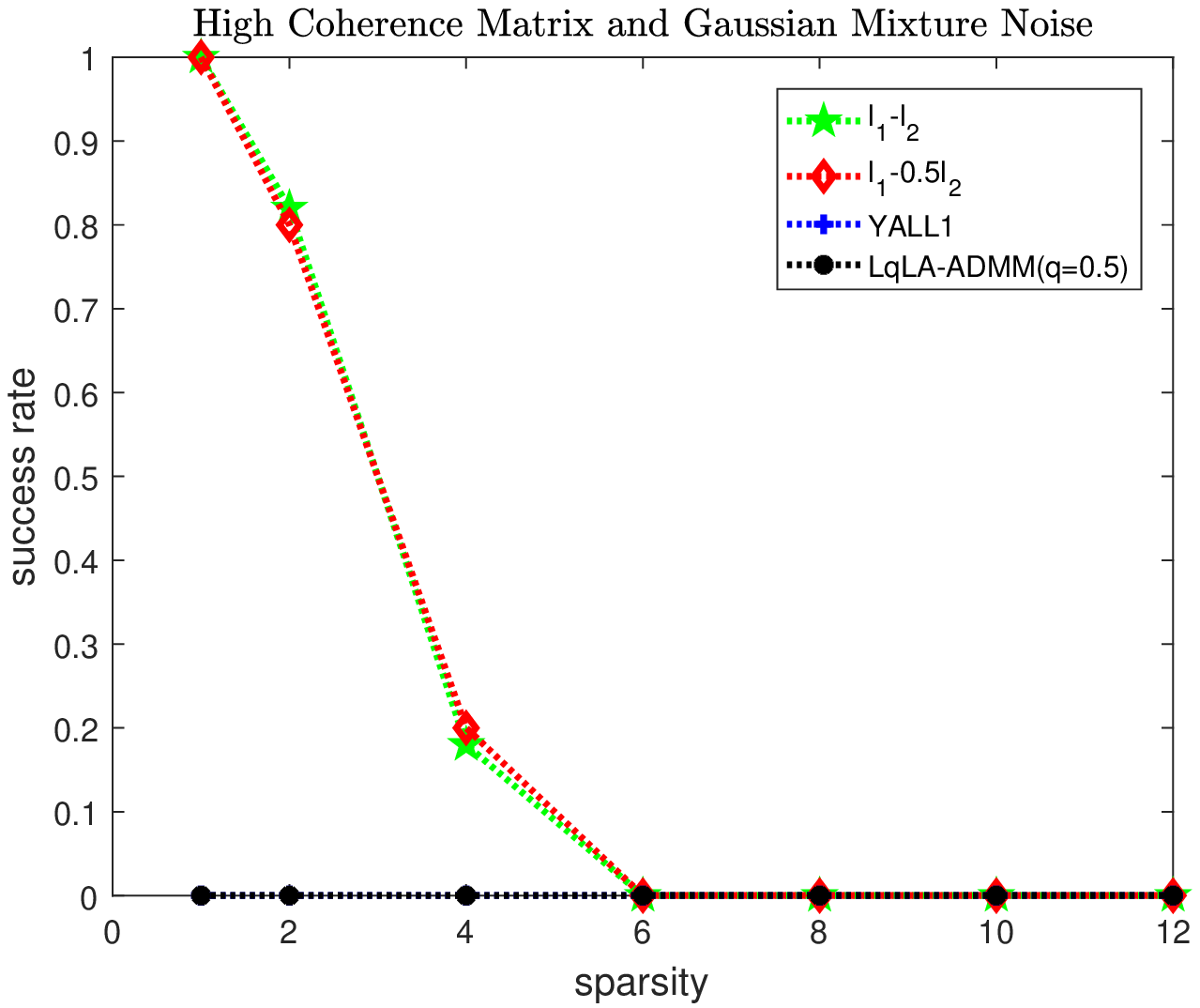}%\\[-0.3cm]
\end{minipage}
\\
\caption{\label{f4.2}Gaussian mixture noise. Left: $m=128, n=
256, s=1,5,10,20,\ldots,80$, $\bm{A}\in\R^{m\times n}$ has small coherence with $\mu(\bm{A})<0.35$; Middle: $m=64, n=1024, s=1,2,5,10,15,\ldots,30$,   $\bm{A}\in\R^{m\times n}$ has mild coherence with $0.5<\mu(\bm{A})<0.6$;
Right:  $m=32, n=640, s=1,2,4,6,\ldots,12$, $\bm{A}\in\R^{m\times n}$ has  high coherence with $\mu(\bm{A})>0.99$.}	
\end{figure*}

In our experiments, let  $\bm{x}\in\R^{n}$ be a simulated $s$-sparse signal,
 where the support of $\bm{x}$ is a random index set and
  the $s$ non-zeros entries obey the Gaussian distribution $\mathcal{N}(0,1)$.
we evaluate the compared methods using simulated sparse signals in various noise conditions.
%We use a simulated $s$-sparse signal of length $n$, in which the positions of the $s$ non-zeros are uniformly randomly chosen while the amplitude of each nonzero entry is generated according to the Gaussian distribution $\mathcal{N}(0,1)$.
In addition, the signal $\bm{x}$ is normalized to have a unit energy value.
Let $\bm{\hat{\bm{x}}}$ be a reconduction of $\bm{x}$ by apply each solver (YALL1 \cite{YZ2011}, LqLA-ADMM$(0<q<1)$ \cite{WPYYL2017} and  proposed $\ell_1-\alpha\ell_2$LA).
If
$$\|\hat{\bm{x}}-\bm{x}\|_2/\|\bm{x}\|_2\leq 10^{-2},$$
the reconstruction is a success.
%A recovery $\hat{\bm{x}}$ is regarded as successful if the relative error satisfies
%$\|\hat{\bm{x}}-\bm{x}^0\|_2/\|\bm{x}^0\|_2\leq 10^{-2}$, where $\bm{x}^0$ is the original signal.
 Each provided result is an average over 100 independent Monte Carlo runs.
 %We compare our proposed method with YALL1, LqLA-ADMM$(0<q<1)$.

For both $S\tau S$ noise and Gaussian mixture noise, we respectively
 design three experiments.
In the first experiment, the  sensing
matrix $\bm{A}\in \R^{m\times n}$ is   orthonormal Gaussian random matrix with $m=128, n=256$,
 which has small coherence smaller than 0.35.
 In the second experiment, let $m=64$, $n=1280$ and
 %we use a simulated $s$-sparse signal of length $n=1280$, while $m\times n$
 the sensing matrix $\bm{A}\in \R^{m\times n}$  be  orthonormal Gaussian random matrix, which has mild coherence between 0.5 and 0.65.
 In the third experiment, let the  sensing
matrix $\bm{A}\in \R^{m\times n}$  be  oversampled partial DCT matrix with $m=32$ and $n=640$, and it  has high coherence larger than 0.99.

Fig. \ref{f4.1} presents the successful rates of recovery  for the YALL1, the LqLA-ADMM$(q=0.5)$ and  the proposed $\ell_1-\alpha\ell_2$LA ($\alpha=1, 0.5$) versus the sparsity $s$ in the $S\tau S$ noise case
with $\tau=1$ (Cauchy noise) and $\gamma=10^{-4}$.% Here, we take $S\tau S$ noise with $\alpha=1$ (Cauchy noise) and $\gamma=10^{-4}$.
In the left figure of  Fig. \ref{f4.1}, we observe
%In the first experiments,
the LqLA-ADMM$(q=0.5)$ has the best performance, followed by YALL1.
In the middle  figure of  Fig. \ref{f4.1}, the LqLA-ADMM$(q=0.5)$ still has the best performance. But, the difference between LqLA-ADMM$(q=0.5)$
and $\ell_1-\alpha\ell_2$LA becomes smaller.
 However, in the right figure, $\ell_{1}-\alpha\ell_{2}$-PLAD is the best and provides the robust performance regardless of large coherence of $\bm{A}$. And the LqLA-ADMM$(q=0.5)$ and YALL1 have lost efficiency.

And Fig. \ref{f4.2} presents the successful rates of recovery of the compared algorithms versus sparsity $s$ in Gaussian mixture noise with $\xi=0.1$ and $\kappa=1000$. In Fig. \ref{f4.2}, we observe  the same conclusions for
this case as that in $S\tau S$ noise.

\subsection{MRI Reconstruction\label{s5.2}}
\hskip\parindent

In this subsection, we present a two-dimensional example of the reconstruction for  MRI from a limited number of projections.
It was first introduced in \cite{CRT2006-1} to demonstrate the success of compressed sensing. The signal/image is a Shepp-Logan phantom of size $256\times 256$. See Fig. \ref{f5.1}. In this case, the gradient of the signal is
sparse. Thus \cite{CRT2006-1,LDP2007} proposed a model to minimize the (isotropic) total variation (TV) \cite{ROF1992}, i.e.,
\begin{align}\label{TV-BP}
\min\|\bm{u}\|_{TV}~~~\text{subject~ to}~ ~\bm{R}\mathcal{F}\bm{u}=\bm{b},
\end{align}
where $\|\bm{u}\|_{TV}=\|\sqrt{|\mathcal{D}_x\bm{u}|^2+|\mathcal{D}_y\bm{u}|^2}\|_1$ with $\mathcal{D}_x, \mathcal{D}_y$ respectively denoting the horizontal and vertical partial derivative operators, %Here and below
$\mathcal{F}$  is  the Fourier transform, $\bm{R}$ is the sampling mask in the frequency
space, and $\bm{b}$ is the data.  It is claimed in \cite{CRT2006-1} that 22 projections are necessary to
achieve  exact recovery. Later, some works suggest that imposing nonconvex metrices on gradients can achieve exact recovery from fewer numbers of projections, for example $\ell_q~(0<q<1)$ \cite{C2007} using 10 projections, truncated $\ell_{1}$ \cite{GY2012} using 8 projections. More results about MRI reconstruction, readers can refer to \cite{CE2005,NNT2010,CNZ2012,NNT2013,ZBN2017} and so on.

Recently,  Lou, et.al. \cite{LZOX2015} proposed the following weighted difference of convex regularization
\begin{align}\label{L1alphaL2TV-Lasso}
\min&\bigg(\|\mathcal{D}_x \bm{u}\|_1+\|\mathcal{D}_y\bm{u}\|_1-\alpha\Big\|\sqrt{|\mathcal{D}_x\bm{u}|^2
+|\mathcal{D}_y\bm{u}|^2}\Big\|_1\bigg)\nonumber\\
&+\frac{\mu}{2}\|\bm{R}\mathcal{F}\bm{u}-\bm{b}\|_2^2,
\end{align}
where $\alpha\in(0,1]$ is a parameter for a more general model. This model was called $\ell_1-\alpha\ell_2$-TV \cite{LZOX2015}. When $\alpha=1$, (\ref{L1alphaL2TV-Lasso}) is the $\ell_{1-2}$-TV model in
\cite{YLHX2015}. These results of \cite{YLHX2015,LZOX2015} demonstrated that 8 projections are enough to guarantee
exact recovery using $\ell_1-\alpha\ell_2$. However, this model is only fit for Gaussian noise. For impulsive noise, we consider the following model
\begin{align}\label{L1-alphaL2TV-PLAD}
\min&\lambda\bigg(\|\mathcal{D}_x \bm{u}\|_1+\|\mathcal{D}_y\bm{u}\|_1-\alpha\Big\|\sqrt{|\mathcal{D}_x\bm{u}|^2
+|\mathcal{D}_y\bm{u}|^2}\Big\|_1\bigg)\nonumber\\
&+\|\bm{R}\mathcal{F}\bm{u}-\bm{b}\|_1,
\end{align}
where $\bm{b}=\bm{R}\mathcal{F}\bm{u}+\bm{z}$ with noise $\bm{z}\in\mathbb{R}^{n_1\times n_2}$. We call it as $\ell_1-\alpha\ell_2$TV-PLAD. Here, let impulsive noise be $S\tau S$ noise.

By ADMM and DCA algorithms, we present the special algorithm to compute (\ref{L1-alphaL2TV-PLAD}).
Splitting the term $\|\bm{R}\mathcal{F}\bm{u}-\bm{b}\|_1$, and respectively  replacing $\mathcal{D}_x\bm{u}, \mathcal{D}_y\bm{u}$ by $\bm{d}_x,\bm{d}_y$, then one has an equivalent problem of (\ref{L1-alphaL2TV-PLAD}) as
follows
\begin{align}\label{L1alphaL2TV-LAD-Equivalent}
\min&~\lambda\bigg(\|\bm{d}_x\|_1+\|\bm{d}_y\|_1-\alpha\Big\|\sqrt{|\bm{d}_x|^2+|\bm{d}_y|^2}\Big\|_1\bigg)
+\|\bm{v}\|_1\nonumber\\
\text{s.~t.}&~\mathcal{D}_x\bm{u}=\bm{d}_x,~\mathcal{D}_y\bm{u}=\bm{d}_y,~\bm{R}\mathcal{F}\bm{u}-\bm{v}=\bm{b}.
\end{align}
Let
\begin{align*}
&\mathcal{L}(\bm{u},\bm{v},\bm{d}_x,\bm{d}_y;\bm{w},\bm{h}_x,\bm{h}_y)\nonumber\\
&=\lambda\bigg(\|\bm{d}_x\|_1+\|\bm{d}_y\|_1
-\alpha\Big\|\sqrt{|\bm{d}_x|^2+|\bm{d}_y|^2}\Big\|_1\bigg) +\|\bm{v}\|_1\nonumber\\
&\hspace*{12pt}+\frac{\rho_1}{2}\|\bm{R}\mathcal{F}\bm{u}-\bm{v}-\bm{b}\|_2^2-\langle \bm{w}, \bm{R}\mathcal{F}\bm{u}-\bm{v}-\bm{b}\rangle \nonumber\\
&\hspace*{12pt}+\frac{\rho_2}{2}\|\mathcal{D}_x\bm{u}-\bm{d}_x\|_2^2
-\langle \bm{h}_x, \mathcal{D}_x\bm{u}-\bm{d}_x\rangle\nonumber\\
&\hspace*{12pt}+\frac{\rho_2}{2}\|\mathcal{D}_y\bm{u}-\bm{d}_y\|_2^2-\langle \bm{h}_y, \mathcal{D}_y\bm{u}-\bm{d}_y\rangle
\end{align*}
be the augmented Lagrangian function of (\ref{L1alphaL2TV-LAD-Equivalent}) with the Lagrangian multipliers $\bm{w},\bm{h}_x,\bm{h}_y\in\mathbb{R}^{n_1\times n_2}$.
Then using ADMM iterate scheme and DCA in $\bm{d}_x,\bm{d}_y$-subproblem, we give the special algorithm.

\begin{algorithm}\label{al:LADMg}
\centering
\caption{$\ell_1-\alpha\ell_2$LA for solving $\ell_{1}-\alpha\ell_{2}$TV-PLAD-(\ref{L1-alphaL2TV-PLAD})}
\vspace{-0mm}
\begin{tabular}{@{}ll}
%\hline
\textbf{Input}~~$\bm{R}$, $\bm{b}$, $0<\alpha\leq 1$,  $\lambda$,  $\rho_1,\rho_2$.\\
\textbf{Initialize}~~\\
~~~~~~$(\bm{u},\bm{v},\bm{d}_x,\bm{d}_y;\bm{w},\bm{h}_x,\bm{h}_y)
=(\bm{u}^0,\bm{v}^0,\bm{d}_x^0,\bm{d}_y^0;\bm{w}^0,\bm{h}_x^0,\bm{h}_y^0)$,\\
~~~~~~~$k=0$.\\
\textbf{While}~~some stopping criterion is not satisfied \textbf{do}\\
~~\textbf{1.}~~Compute sub-gradient $\bm{q}^{k}$ of $\|\sqrt{|\bm{d}_x|^2+|\bm{d}_y|^2}\|_1$  at point\\
 ~~~~~~$(\bm{d}_x^k,\bm{d}_y^k)$ by \\
~~~~~~$\bm{q}^{k}=(\bm{q}_x^k;\bm{q}_y^k)=\frac{(\bm{d}_x^k;\bm{d}_y^k)}
 {\sqrt{|\bm{d}_x^k|^2+|\bm{d}_y^k|^2}}$\\

~~\textbf{2.}~~
Compute $\bm{u}^{k+1}$ by\\
~~~~~~$\bm{u}^{k+1}=\big(\rho_1\bm{R}^{T}\bm{R}-\rho_2\mathcal{\triangle}\big)^{-1}
\bigg(\rho_1\mathcal{F}^{*}\bm{R}(\bm{b}+\bm{v}^k+\bm{w}^k)$\\
~~~~~~$+\rho_2 \mathcal{D}_x^{T}(\bm{d}_x^k+\bm{h}_x^k)+\rho_2 \mathcal{D}_y^{T}(\bm{d}_y^k+\bm{h}_y^k)\bigg).$\\

~~\textbf{3.}~~Compute $\bm{v}^{k+1}$ by\\
~~~~~~$\bm{v}^{k+1}=S(\bm{R}\mathcal{F}\bm{u}^{k+1}-\bm{b}-\bm{w}^k,\frac{1}{\rho_1})$.\\
~~\textbf{4.}~~Update $\bm{d}_x^{k+1},~\bm{d}_y^{k+1}$ via\\
~~~~~~$\bm{d}_x^{k+1}=S\bigg((\mathcal{D}_x\bm{u}^{k+1}
-\bm{h}_x^k)+\frac{\lambda\alpha}{\rho_2} \bm{q}_x^k,\frac{\lambda}{\rho_2}\bigg)$,\\
~~~~~~ $\bm{d}_y^{k+1}=S\bigg((\mathcal{D}_y\bm{u}^{k+1}-\bm{h}_y^k)+\frac{\lambda\alpha}{\rho_2} \bm{q}_y^k,\frac{\lambda}{\rho_2}\bigg)$.\\
~~\textbf{5.}~~Update dual variables \\
~~~~~~$
\bm{w}^{k+1}=\bm{w}^k-\rho_1(\bm{R}\mathcal{F}\bm{u}^{k+1}-\bm{v}^{k+1}-\bm{b})$\\
~~~~~~$\bm{h}_x^{k+1}=\bm{h}_x^k-\rho_2(\mathcal{D}_x\bm{u}^{k+1}-\bm{d}_x^{k+1})$\\
~~~~~~$\bm{h}_y^{k+1}=\bm{h}_y^k-\rho_2(\mathcal{D}_y\bm{u}^{k+1}-\bm{d}_y^{k+1})$.\\
~~\textbf{6.}~~$k=k+1$.\\

\textbf{~~End}
\vspace{-3.5pt} \\
%\hline
\end{tabular}
\end{algorithm}

\begin{remark}
In Algorithm 2,
$\alpha$ is a model parameter  and satisfies  $0<\alpha\leq 1$, $\lambda>0$ is a
 penalty parameter, $0<\rho_1,~\rho_2<1$ are regularized parameters.
\end{remark}

In this section, numerical experiments
compare our $\ell_1-\alpha \ell_2$LA algorithm   with some other efficient methods  including YALL1 \cite{YZ2011} for penalized LAD model
\begin{align}\label{L1TV-PLAD}
\min~\lambda(\|\mathcal{D}_x \bm{u}\|_1+\|\mathcal{D}_y\bm{u}\|_1)+\|\bm{R}\mathcal{F}\bm{u}-\bm{b}\|_1
\end{align}
and LqLA-ADMM \cite{WPYYL2017}
\begin{align}\label{LqTV-PLAD}
\min~\lambda(\|\mathcal{D}_x \bm{u}\|_q^q+\|\mathcal{D}_y\bm{u}\|_q^q)
+\|\bm{R}\mathcal{F}\bm{u}-\bm{b}\|_{1,\varepsilon},
\end{align}
where $0<q<1$.

Fig. \ref{f5.1} shows the stable recovery of 8 projections using the proposed method. In Fig. \ref{f5.1}, the root-mean-square (RMS) error is used to measure the performance quantitatively. The RMS between reference and distorted images $\bm{X}$, $\bm{Y}$ is defined
as $\text{RMS}(\bm{X},\bm{Y})=\|\bm{X}-\bm{Y}\|_2/\sqrt{M}$, where $M$ is the number of pixels in images $\bm{X}$,$\bm{Y}$. Figure \ref{f5.1}  explains that $\ell_1-\alpha \ell_2~(\alpha=0.5)$ is much better than YALL1 and LqLA-ADMM~$(q=0.5)$ visually as well as in terms of RMS. Fig. \ref{f5.1} also shows that 8 projections are sufficient to have stable recovery in impulsive noise by using the $\ell_1-\alpha \ell_2$LA method.

\begin{figure*}[htbp!]
\begin{centering}
\includegraphics[width=15.0cm]{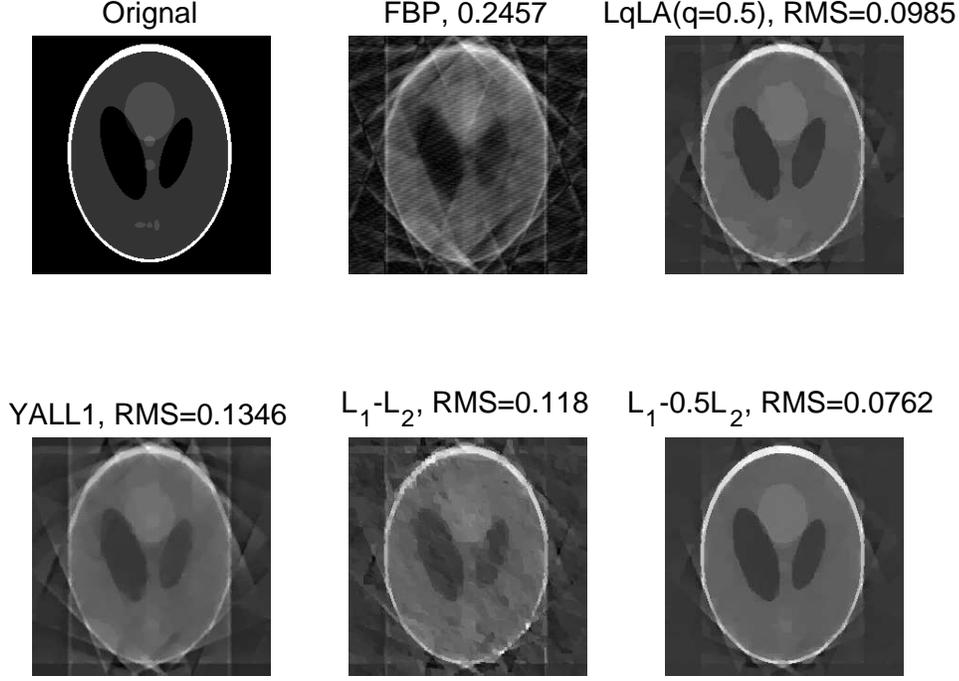}
\par\end{centering}
\caption{MRI reconstruction from observation with impulsive noise. It is demonstrated that 8 projections are enough to have stable recovery in impulsive using $L_1-\alpha L_2$LA. The root-mean-square (RMS) errors are provided for each method.}
\label{f5.1}	
\end{figure*}

%%%%%%%%%%%%%%%%%%%%%%%%%%%%%%%%%%%%%%%%%%%%%%%%%%%%%%%%%%%%%%%%%%%%
%%%%%%%%%%%%%%%%%%%%%%   section 6 %%%%%%%%%%%%%%%%%%%%%%%%%%%%%%%%%
%%%%%%%%%%%%%%%%%%%%%%%%%%%%%%%%%%%%%%%%%%%%%%%%%%%%%%%%%%%%%%%%%%%
\section{Conclusions \label{s6}}
\hskip\parindent

In this paper, we consider the signal and image reconstructions in impulsive noise via $\ell_{1}-\alpha\ell_{2}~(0<\alpha\leq 1)$ minimization. First, we propose the two new  models of  $\ell_{1}-\alpha\ell_{2}$-LAD (\ref{VectorL1-alphaL2-LAD}), and $\ell_{1}-\alpha\ell_{2}$-DS (\ref{VectorL1-alphaL2-DS}) in Section \ref{s1}. In Section \ref{s2},  we obtain a sufficient  condition based on $(\ell_2,\ell_1)$-RIP to guarantee the exact recovery of $\bm{x}$ from $\bm{b}=\bm{A}\bm{x}$ via (\ref{VectorL1-alphaL2-Exact}) (see Theorem \ref{ExactRecoveryviaVectorL1-alphaL2-Exact})).
And in Section \ref{s3}, we consider the recovery of $\bm{x}$ via (\ref{VectorL1-alphaL2-LAD}) and (\ref{VectorL1-alphaL2-DS})  in the noisy case. We  give the sufficient $(\ell_2,\ell_1)$-RIP conditions to guarantee the stable recovery of $\bm{x}$ from $\bm{b}=\bm{A}\bm{x}+\bm{z}$ (see Theorem \ref{StableRecoveryviaVectorL1-alphaL2-LAD} and Theorem \ref{StableRecoveryviaVectorL1-alphaL2-DS}).

In order to obtain the efficient algorithm of (\ref{VectorL1-alphaL2-LAD}), %compute the proposed model,
we  introduce the unconstrained $\ell_{1}-\alpha\ell_{2}$ model $\ell_{1}-\alpha\ell_{2}$-PLAD (\ref{VectorL1-alphaL2-PLAD}).  Using ADMM and DCA,
we have developed a numerical scheme-$\ell_{1}-\alpha\ell_{2}$LA to efficiently solve our unconstrained problem (\ref{VectorL1-alphaL2-PLAD}) in section \ref{s4}.

Last, we present numerical experiments for the  sparse signal and compressible image recovery in impulsive noise case. They  demonstrate the efficiency of $\ell_{1}-\alpha\ell_{2}$LA method (see section \ref{s5}). In signal recovery experiments,  let sensing matrix $\bm{A}$ has different coherence: small coherence $\mu(\bm{A})<0.35$, mild coherence $0.5<\mu(\bm{A})<0.65$ and high coherence $\mu(\bm{A})>0.99$. Although our method performs not well when sensing matrix has small coherence, the difference is smaller when the coherence increases. And when the measurement matrix has high coherence, our method becomes the best. And the MRI phantom image recovery test also demonstrates that $\ell_1-\alpha \ell_2$LA is highly effective and comparable to state-of-the-art methods.

%We also notice that \cite{MLH2017} considered to recover the low-rank matrices or matrix completion via truncated $\ell_{1-2}$ minimization.  Therefore, the future works will include sparse signal recovery, low-rank matrix recovery and compressible image reconstruction in impulsive noise via truncated $\ell_{1}-\alpha\ell_{2}$ minimization.

%%%%%%%%%%%%%%%%%%%%%%%%%%%%%%%%%%%%%%%%%%%%%%%%%%%%%%%%%%%%%%%%
%%%%%%%%%%%%%%%%%%%%%%%  Section 6 %%%%%%%%%%%%%%%%%%%%%%%%%%%%%%
%%%%%%%%%%%%%%%%%%%%%%%%%%%%%%%%%%%%%%%%%%%%%%%%%%%%%%%%%%%%%%%%

\appendices

 \section{Proof of Theorem~\ref{ExactRecoveryviaVectorL1-alphaL2-Exact}}\label{pro:ExactRecoveryviaVectorL1-alphaL2-Exact}
\begin{proof}
Let $\hat{\bm{x}}$ be the minimizer of \eqref{VectorL1-alphaL2-Exact}.
Clearly,
$\bm{b}=\bm{A}\hat{\bm{x}}$ and $\|\hat{\bm{x}}\|_{\alpha,1-2}\leq \|\bm{x}\|_{\alpha,1-2}$.
Let $\bm{h}=\hat{\bm{x}}-\bm{x}$. Suppose that $\bm{h}\in\mathcal{N}(\bm{A})\backslash\{\bm{0}\}$. Then by \eqref{e:h-maxsupperboundnoiseless1} in Lemma \ref{ConeconstraintinequalityforL1-L2}, we have % a modified cone constraint inequality as follows
\begin{align}\label{e2.1}
\|\bm{h}_{-\max(s)}\|_1\leq\|\bm{h}_{\max(s)}\|_1+\alpha\|\bm{h}\|_2.
\end{align}
%And we also need a tube constraint inequality
From  $\bm{b}=\bm{Ax}$ and $\bm{b}=\bm{A}\hat{\bm{x}}$, it follows that
\begin{align}\label{e2.8}
\|\bm{Ah}\|_1=\|\bm{A}\hat{\bm{x}}-\bm{Ax}\|_1%=\|b-b\|_1
=\bm{0}.
\end{align}

Let $T_0=\text{supp}(\bm{h}_{\max(s)})$, $t=ks\in\mathbb{Z}_+$,
$T_1$ be the index set of the $t\in\mathbb{Z}_+$
 largest entries of $\bm{h}_{-\max(s)}$  and $T_{01}=T_0\cup T_1$.
 Thus, by the facts that $\bm{A}$ satisfies  the $(\ell_2,\ell_1)$-RIP condition of $(k+1)s$ order,
 $t=ks$, $\bm{x}$ is $s$-sparse  and Lemma \ref{LowerBound},
ones have  a lower bound of $\|\bm{Ah}\|_1$
\begin{align}\label{e2.7}
\|\bm{Ah}\|_1
&\geq\rho_{ks}\|\bm{h}_{T_{01}}\|_2,
\end{align}
where $\rho_{ks}=1-\delta_{(k+1)s}^{lb}-\frac{(1+\delta_{ks}^{ub})}{a(s,ks;\alpha)}$
with $a(s,ks;\alpha)=\frac{\sqrt{ks}-\alpha}{\sqrt{s}+\alpha}>1.$
% We partition $T_0^c=[[1,n]]\backslash T_0$ as
%$$
%T_0^c=\cup_{j=1}^J T_j,
%$$
%which is the same as the proof of Lemma \ref{LowerBound}.

%Let $T_{01}=T_0\cup T_1$, we consider the following identity
%\begin{align}\label{e2.3}
%\|Ah\|_1=\bigg\|A(h_{T_{01}})+\sum_{j\geq 2}A(h_{T_{j}})\bigg\|_1,
%\end{align}

%First, we give out a lower bound for (\ref{e2.3}). Note that $\|x\|_0\leq s$ implies that $\|x_{-\max(s)}\|_1=0$. Then by Lemma \ref{LowerBound}, ones get
%\begin{align}\label{e2.7}
%\|Ah\|_1
%&\geq\Big(1-\delta_{t+s}^{lb}-\frac{(1+\delta_{t}^{ub})}{a(s,t;\alpha)}\Big)\|h_{T_{01}}\|_2,
%\end{align}
%which gives a lower bound of $\|Ah\|_1$.

%We also need an upper bound of $\|Ah\|_1$. Using (\ref{e2.2}), we have
%\begin{align}\label{e2.8}
%\|Ah\|_1=0.
%\end{align}

Combining the lower bound (\ref{e2.7}) with %the upper bound
(\ref{e2.8}),
 we have
\begin{align}\label{e2.9}
0\geq\Big(1-\delta_{(k+1)s}^{lb}-\frac{1+\delta_{ks}^{ub}}{a(s,ks;\alpha)}\Big)\|\bm{h}_{T_{01}}\|_2.
\end{align}
%Taking $t=ks\in\mathbb{Z}_+$ such that $a(s;\alpha,k):=a(s;ks,\alpha)=\frac{\sqrt{ks}-\alpha}{\sqrt{s}+\alpha}>1$ .
Note that
%$\delta_{ks}^{ub}+a(s;\alpha)\delta_{(k+1)s}^{lb}<a(s;\alpha)-1$
the condition \eqref{e2.0} implies that
$$
1-\delta_{(k+1)s}^{lb}-\frac{1+\delta_{ks}^{ub}}{a(s,ks;\alpha)}>0,
$$
i.e., $\rho_{ks}>0$.
Then by (\ref{e2.9}), it is clear that
$$
\|\bm{h}_{T_{01}}\|_2\leq0,
$$
Therefore, by the definition of  $T_{01}$,
 $\bm{h}=\bm{0}$, which contradicts  with the assumption $\bm{h}\in\mathcal{N}(\bm{A})\backslash\{\bm{0}\}$.
 We complete the proof.
\end{proof}

\section{Proof of Theorem~\ref{StableRecoveryviaVectorL1-alphaL2-LAD}}\label{pro:StableRecoveryviaVectorL1-alphaL2-LAD}
\begin{proof}
Let $\bm{h}=\hat{\bm{x}}^{\ell_{1}}-\bm{x}$.
Since $\hat{\bm{x}}^{\ell_{1}}$ is the minimizer of (\ref{VectorL1-alphaL2-LAD}),
 $\|\bm{b}-\bm{A}\hat{\bm{x}}^{\ell_1}\|_1\leq \eta_1$ and
$\|\hat{\bm{x}}^{\ell_{1}}\|_{\alpha,1-2}\leq \|\bm{x}\|_{\alpha,1-2}$.
Then, by \eqref{e:h-maxsupperbound2} in Lemma \ref{ConeconstraintinequalityforL1-L2}, %we have a modified cone constraint inequality as follows
we have
\begin{align}\label{e3.2}
\|\bm{h}_{-\max(s)}\|_1\leq\|\bm{h}_{\max(s)}\|_1+2\|\bm{x}_{-\max(s)}\|_1+\alpha\|\bm{h}\|_2.
\end{align}
By the facts that $\|\bm{z}\|_1=\|\bm{b}-\bm{Ax}\|_1\leq \eta_1$
and $\|\bm{b}-\bm{A}\hat{\bm{x}}^{\ell_1}\|_1\leq \eta_1$,
one has % the tube constraint inequality
%And we also need a tube constraint inequality
\begin{align}\label{e3.3}
\|\bm{Ah}\|_1=\|\bm{A}\hat{\bm{x}}^{\ell_1}-\bm{Ax}\|_1
\leq&\|\bm{A}\hat{\bm{x}}^{\ell_1}-\bm{b}\|_1\nonumber\\
+\|\bm{b}-\bm{Ax}\|_1\leq&\eta_1+\eta_1=2\eta_1.
\end{align}

Similar to  the proof of Theorem \ref{ExactRecoveryviaVectorL1-alphaL2-Exact},
let $T_0=\text{supp}(\bm{h}_{\max(s)})$, $t=ks\in\mathbb{Z}_+$,
$T_1$ be the index set of the $t=ks$
 largest entries of $\bm{h}_{-\max(s)}$  and $T_{01}=T_0\cup T_1$.
 Thus, by the facts that $\bm{A}$ satisfies  the $(\ell_2,\ell_1)$-RIP condition of $(k+1)s$ order,
 $t=ks$, %$\bm{x}$ is $s$-sparse
 and Lemma \ref{LowerBound},
ones obtain a lower bound of $\|\bm{Ah}\|_1$
%a lower bound of $\|\bm{Ah}\|_1$
%Let $T_0=\text{supp}(h_{\max(s)})$. We also partition $T_0^c=[[1,n]]\backslash T_0$ as
%$$
%T_0^c=\cup_{j=1}^J T_j,
%$$
%which is the same as the proof of Theorem \ref{ExactRecoveryviaVectorL1-L2}.
%In order to estimate $\|h_{T_{01}}\|_2$, we still consider identity (\ref{e2.3}).

%First, we estimate a lower bound for (\ref{e2.3}). It is clear to claim that (\ref{e2.4}) still hold. We also need an upper bound for $\sum_{j\geq 2}\|h_{T_j}\|_{2}$. Note that for $j\geq 2$, it follows from (\ref{e3.2}) that
%\begin{align}\label{e3.4}
%\sum_{j\geq 2}\|h_{T_j}\|_2&\leq\frac{\sum_{j\geq2}(\|h_{T_{j-1}}\|_1-\alpha\|h_{T_{j-1}}\|_2)}{\sqrt{t}-\alpha}\nonumber\\
%\leq&\frac{\|h_{T_{0}^c}\|_1-\alpha\Big(\sum_{j\geq2}\|h_{T_{j-1}}\|_2^2\Big)^{1/2}}{\sqrt{t}-\alpha}\nonumber\\
%&=\frac{\|h_{T_{0}^c}\|_1-\alpha\|h_{T_0^c}\|_2}{\sqrt{t}-\alpha}\nonumber\\
%&\leq\frac{\|h_{T_{0}}\|_1+2\|x_{-\max(s)}\|_1+\alpha\|h_{T_0}\|_2}{\sqrt{t}-\alpha}\nonumber\\
%&\leq\frac{(\sqrt{s}+\alpha)\|h_{T_0}\|_2}{\sqrt{t}-\alpha}+\frac{2\|x_{-\max(s)}\|_1}{\sqrt{t}-\alpha}\nonumber\\
%:=&\frac{\|h_{T_0}\|_2}{a(s,t;\alpha)}+\frac{2\|x_{-\max(s)}\|_1}{\sqrt{t}-\alpha},
%\end{align}
%instead of (\ref{e2.5}). Then by inserting (\ref{e3.4}) into (\ref{e2.4}), we have
\begin{align}\label{e3.5}
\|\bm{Ah}\|_1
%&\geq (1-\delta_{t+s}^{lb})\|h_{T_{01}}\|_{2}-(1+\delta_{t}^{ub})\frac{\|h_{T_0}\|_2}{a(s,t;\alpha)}
%-(1+\delta_{t}^{ub})\frac{2\|x_{-\max(s)}\|_1}{\sqrt{t}-\alpha}\nonumber\\
\geq&\rho_{ks}\|\bm{h}_{T_{01}}\|_{2}
-(1+\delta_{ks}^{ub})\frac{2\|\bm{x}_{-\max(s)}\|_1}{\sqrt{ks}-\alpha},
\end{align}
%which presents  a lower bound of $\|Ah\|_1$.
where $\rho_{ks}=1-\delta_{(k+1)s}^{lb}-\frac{(1+\delta_{ks}^{ub})}{a(s,ks;\alpha)}$
with $a(s,ks;\alpha)=\frac{\sqrt{ks}-\alpha}{\sqrt{s}+\alpha}$.
%We also need an upper bound of $\|Ah\|_1$. Using (\ref{e3.3}), we have
%\begin{align}\label{e3.6}
%\|Ah\|_1\leq2\eta_1.
%\end{align}

%Combining %the lower bound
By (\ref{e3.5}) and
%the upper bound
(\ref{e3.3}), we have
\begin{align}\label{e3.7}
2\eta_1\geq&\Big(1-\delta_{(k+1)s}^{lb}-\frac{(1+\delta_{ks}^{ub})}{a(s, ks; \alpha)}\Big)\|\bm{h}_{T_{01}}\|_{2}\nonumber\\
&-(1+\delta_{ks}^{ub})\frac{2\|\bm{x}_{-\max(s)}\|_1}{\sqrt{ks}-\alpha}.
\end{align}
%Taking $t=ks\in\mathbb{Z}_+$ such that
where $a(s;ks,\alpha)=\frac{\sqrt{ks}-\alpha}{\sqrt{s}+\alpha}>1$.
Furthermore, the condition  \eqref{e2.0}
 %Note that
%$\delta_{ks}^{ub}+a(s;\alpha,k)\delta_{(k+1)s}^{lb}<a(s;\alpha,k)-1$
implies that
$$
1-\delta_{(k+1)s}^{lb}-\frac{(1+\delta_{ks}^{ub})}{a(s,ks;\alpha)}>0,
$$
that is $\rho_{ks}>0$.
Then, by  (\ref{e3.7}), one has
\begin{align}\label{e3.8}
\|\bm{h}_{T_{01}}\|_{2}\leq\frac{2}{\rho_{ks}}\eta_1
+\frac{(1+\delta_{ks}^{ub})\sqrt{s}}{(\sqrt{ks}-\alpha)\rho_{ks}}\frac{2\|\bm{x}_{-\max(s)}\|_1}{\sqrt{s}}.
\end{align}
%where $\rho_1=1-\delta_{(k+1)s}^{lb}-\frac{(1+\delta_{ks}^{ub})}{a(s;ks,\alpha)}$.

By the fact that $\|\bm{h}\|_2=\sqrt{\|\bm{h}_{T_{01}}\|_2^2+\|\bm{h}_{T_{01}^c}\|_2^2}$,
 to show \eqref{e:stablel1upper}, we need to
 estimate the upper bound of $\|\bm{h}_{T_{01}^c}\|_{2}$.
 Without loss of generality, we assume that $|h_1|\geq\cdots\geq|h_s|\geq|h_{s+1}|\geq\cdots\geq|h_{s+t}|\geq\cdots\geq|h_n|$ with $t=ks\in \mathbb{Z}_+$.
 Then,
%It follows from (\ref{e3.2}) that
\begin{align}\label{e3.9}
\|\bm{h}_{T_{01}^c}\|_{2}
\leq&\sqrt{\|\bm{h}_{T_{01}^c}\|_1\|\bm{h}_{T_{01}^c}\|_{\infty}}\nonumber\\
\overset{(1)}{\leq}&\sqrt{\Big(\|\bm{h}_{T_0^c}\|_1-\sum_{j\in T_1}|h_j|\Big)|h_{s+t}|}\nonumber\\
\overset{(2)}{\leq}&\sqrt{\Big(\|\bm{h}_{T_0^c}\|_1-t|h_{s+t}|\Big)|h_{s+t}|}\nonumber\\
=&\sqrt{-t\Big(|h_{s+t}|-\frac{\|\bm{h}_{T_0^c}\|_1}{2t}\Big)^2
+\frac{\|\bm{h}_{T_0^c}\|_1^2}{4t}}\nonumber\\
&\leq\frac{\|\bm{h}_{T_0^c}\|_1}{2\sqrt{t}}\nonumber\\
\overset{(3)}{\leq}&\frac{\|\bm{h}_{\max(s)}\|_1+2\|\bm{x}_{-\max(s)}\|_1+\alpha\|\bm{h}\|_2}{2\sqrt{t}}\nonumber\\
\overset{(4)}{\leq}&\frac{1}{2}\sqrt{\frac{s}{t}}\Big(\|\bm{h}_{T_{01}}\|_2+\frac{2\|\bm{x}_{-\max(s)}\|_1
+\alpha\|\bm{h}\|_2}{\sqrt{s}}\Big)\nonumber\\
\overset{(5)}{=}&\frac{1}{2\sqrt{k}}\bigg(\|\bm{h}_{T_{01}}\|_2+\frac{2\|\bm{x}_{-\max(s)}\|_1
+\alpha\|\bm{h}\|_2}{\sqrt{s}}\bigg),
\end{align}
where (1) and (2) are from $T_{01}=T_0\cup T_1$, $|T_1|=t$ and the assumption $|h_1|\geq\cdots\geq|h_s|\geq|h_{s+1}|\geq\cdots\geq|h_{s+t}|\geq\cdots\geq|h_n|$,
(3) follows from \eqref{e3.2}, (4) is due to $\|\bm{h}_{\max(s)}\|_1\leq \sqrt{s}\|\bm{h}_{\max(s)}\|_2$,
$T_0=\mathrm{supp}(\bm{h}_{\max(s)})$ and $T_{01}=T_0\cup T_1$, and
(5) follows from $t=ks\in \mathbb{Z}_+$.

%Combining with the estimate of $\|h_{T_{01}}\|_2$ with $\|h_{T_{01}^c}\|_2$,
By  \eqref{e3.9}, ones have
\begin{align}\label{e:hupperbound}
&\|\bm{h}\|_2=\sqrt{\|\bm{h}_{T_{01}}\|_2^2+\|\bm{h}_{T_{01}^c}\|_2^2}\nonumber\\
&\leq\sqrt{\|\bm{h}_{T_{01}}\|_2^2+\frac{1}{4k}\bigg(\|\bm{h}_{T_{01}}\|_2+\frac{2\|\bm{x}_{-\max(s)}\|_1
+\alpha\|\bm{h}\|_2}{\sqrt{s}}\bigg)^2}\nonumber\\
&\leq\bigg(1+\frac{1}{2\sqrt{k}}\bigg)\|\bm{h}_{T_{01}}\|_2+\frac{1}{2\sqrt{k}}\frac{2\|\bm{x}_{-\max(s)}\|_1}{\sqrt{s}}
+\frac{1}{2\sqrt{k}}\frac{\alpha\|\bm{h}\|_2}{\sqrt{s}},
\end{align}
where the last inequality is due to the basic inequality $\sqrt{a^2+b^2}\leq a+b$ for $a,b\geq 0$.
Since $0\leq \alpha \leq 1$ and $ks\in \mathbb{Z}_+$, $1-\frac{\alpha}{2\sqrt{ks}}>0$.
Thus, based on \eqref{e:hupperbound}, we have
%which implies that
\begin{align*}%\label{e3.10}
\|\bm{h}\|_2\leq\frac{(2\sqrt{k}+1)\sqrt{s}}{2\sqrt{ks}-\alpha}\|\bm{h}_{T_{01}}\|_2
+\frac{\sqrt{s}}{2\sqrt{ks}-\alpha}\frac{2\|\bm{x}_{-\max(s)}\|_1}{\sqrt{s}}.
\end{align*}
Substituting (\ref{e3.8}) into the above inequality, ones get
\begin{align*}
&\|\bm{h}\|_2
\leq\frac{(2\sqrt{k}+1)\sqrt{s}}{2\sqrt{ks}-\alpha}
\Big(\frac{2}{\rho_{ks}}\eta_1+\frac{(1+\delta_{ks}^{ub})
\sqrt{s}}{\rho_{ks}(\sqrt{ks}-\alpha)}\frac{2\|\bm{x}_{-\max(s)}\|_1}{\sqrt{s}}\Big)\\
&\ \ +\frac{\sqrt{s}}{2\sqrt{ks}-\alpha}\frac{2\|\bm{x}_{-\max(s)}\|_2}{\sqrt{s}}\\
&=\frac{\sqrt{s}}{2\sqrt{ks}-\alpha}\Big(\frac{(2\sqrt{k}+1)(1+\delta_{ks}^{ub})
\sqrt{s}}{\rho_{ks}(\sqrt{ks}-\alpha)}+1\Big)
\frac{2\|\bm{x}_{-\max(s)}\|_1}{\sqrt{s}}\\
&+\frac{2(2\sqrt{k}+1)\sqrt{s}}{(2\sqrt{2}\sqrt{s}-\alpha)\rho_{ks}}\eta_1.
\end{align*}
We complete the proof of Theorem \ref{StableRecoveryviaVectorL1-alphaL2-LAD}.
\end{proof}

\section{Proof of Theorem~\ref{StableRecoveryviaVectorL1-alphaL2-DS}}\label{pro:StableRecoveryviaVectorL1-alphaL2-DS}
\begin{proof}
Take $\bm{h}=\hat{\bm{x}}^{DS}-\bm{x}$.
Since  $\hat{\bm{x}}^{DS}$ is the minimizer  of (\ref{VectorL1-alphaL2-DS}),
which implies $\|\hat{\bm{x}}^{DS}\|_{\alpha,1-2}\leq \|\bm{x}\|_{\alpha,1-2}$
and $\|\bm{A}^{*}(\bm{b}-\bm{A}\hat{\bm{x}}^{DS})\|_\infty\leq \eta_2$,
 \eqref{e3.2} still holds.
%But
From the facts $\|\bm{A}^{*}\bm{z}\|_\infty=\|\bm{A}^{*}(\bm{b}-\bm{Ax})\|_\infty\leq \eta_2$
and $\|\bm{A}^{*}(\bm{b}-\bm{A}\hat{\bm{x}}^{DS})\|_\infty\leq \eta_2$,
we have the following tube constraint inequality
\begin{align}\label{e3.11}
\|\bm{A}^{*}\bm{Ah}\|_{\infty}
=&\|\bm{A}^{*}(\bm{A}^{*}{\bm{x}}^{DS}-\bm{Ax)}\|_{\infty}\nonumber\\
\leq&\|\bm{A}^{*}(\bm{A}\hat{\bm{x}}^{DS}-\bm{b})\|_{\infty}+\|\bm{A}^{*}(\bm{b}-\bm{Ax})\|_{\infty}\nonumber\\
\leq&\eta_2+\eta_2=2\eta_2
\end{align}
instead of (\ref{e3.3}).

Similarly,
let $T_0=\text{supp}(\bm{h}_{\max(s)})$,
$t=ks\in\mathbb{Z}_+$,
$T_1$ be the index set of the $t\in\mathbb{Z}_+$
 largest entries of $\bm{h}_{-\max(s)}$  and $T_{01}=T_0\cup T_1$.
Since $\bm{A}$ satisfies  the $(\ell_2,\ell_1)$-RIP condition of $(k+1)s$ order,
$t=ks$, %$\bm{x}$ is $s$-sparse
and Lemma \ref{LowerBound},
 % We also partition $T_0^c=[[1,n]]\backslash T_0$ as
%$$
%T_0^c=\cup_{j=1}^J T_j,
%$$
%which is the same as the proof of Theorem \ref{ExactRecoveryviaVectorL1-L2}.
%With the same proof of Theorem \ref{StableRecoveryviaVectorL1-L2andL1},
%we still have
(\ref{e3.5}) holds, which presents a lower bound of $\|\bm{Ah}\|_1$.

Next, we  estimate the upper bound of $\|\bm{Ah}\|_1$ using new technology,
which is completely different  from that of the proof for Theorem \ref{StableRecoveryviaVectorL1-alphaL2-LAD}.

%But the upper bound of $\|Ah\|_1$ needs a completely different proof from that of Theorem \ref{StableRecoveryviaVectorL1-L2andL1}.

By Cauchy-Schwartz inequality,  we have
\begin{align}\label{e3.12}
\|\bm{Ah}\|_1
&\leq \sqrt{m}\|\bm{Ah}\|_2
=\sqrt{m}\langle \bm{Ah}, \bm{Ah}\rangle^{1/2}\nonumber\\
&=\sqrt{m}\langle \bm{A}^{*}\bm{Ah}, \bm{h}\rangle^{1/2}
\leq\sqrt{m}\sqrt{\|\bm{A}^{*}\bm{Ah}\|_{\infty}\|\bm{h}\|_1}\nonumber\\
&=\sqrt{m}\sqrt{\|\bm{A}^{*}\bm{Ah}\|_{\infty}(\|\bm{h}_{T_{0}}\|_1+\|\bm{h}_{T_{0}^c}\|_1)}\nonumber\\
&\overset{(1)}{\leq}\sqrt{m}\sqrt{2\eta_2(2\|\bm{h}_{T_{0}}\|_1+2\|\bm{x}_{-\max(s)}\|_1+\alpha\|\bm{h}\|_2)}\nonumber\\
&\overset{(2)}{\leq}\sqrt{2m\sqrt{s}\eta_2\bigg(2\|\bm{h}_{T_{01}}\|_2+\frac{2\|\bm{x}_{-\max(s)}\|_1
+\alpha\|\bm{h}\|_2}{\sqrt{s}}\bigg)}
\end{align}
where (1) is from \eqref{e3.11}, (2) is due to $T_{01}=T_0\cup T_1$ and
 $\|\bm{h}_{T_{0}}\|_1\leq \sqrt{s}\|\bm{h}_{T_{0}}\|_2$ with $|T_{0}|\leq s$.

Combining  (\ref{e3.5}) with  (\ref{e3.12}), we have
\begin{align}\label{e:inequalitynosimiple}
&%\Big(1-\delta_{t+s}^{lb}-\frac{(1+\delta_{t}^{ub})}{a(s;ks,\alpha)}\Big)
\rho_{ks}\|\bm{h}_{T_{01}}\|_{2}
 -\frac{(1+\delta_{ks}^{ub})\sqrt{s}}{\sqrt{ks}-\alpha}\frac{2\|\bm{x}_{-\max(s)}\|_1}{\sqrt{s}}\nonumber\\
&\leq
\sqrt{2m\sqrt{s}\eta_2\bigg(2\|\bm{h}_{T_{01}}\|_2+\frac{2\|\bm{x}_{-\max(s)}\|_1+\alpha\|\bm{h}\|_2}{\sqrt{s}}\bigg)},
\end{align}
where $\rho_{ks}=1-\delta_{t+s}^{lb}-\frac{(1+\delta_{ks}^{ub})}{a(s,ks;\alpha)}$ with
$a(s,ks;\alpha)=\frac{\sqrt{ks}-\alpha}{\sqrt{s}+\alpha}>1$.
Furthermore,
\begin{align*}
 &1-\delta_{(k+1)s}^{lb}-\frac{1+\delta_{ks}^{ub}}{a(s,ks;\alpha)}\\
 &>1-\delta_{(k+1)s}^{lb}-\frac{(1+b(s,k;\alpha))(1+\delta_{ks}^{ub})}{a(s,ks;\alpha)b(s,k;\alpha)}>0
 \end{align*}
 where the first and  last inequalities are  from  $b(s,k;\alpha)=\frac{8(2\sqrt{ks}-\alpha)}{17\alpha(2\sqrt{k}+1)}>0$
 with $0<\alpha\leq 1$ and  $\eqref{RIPCondition3}$, respectively.

To estimate $\|\bm{h}_{T_{01}}\|_{2}$ from \eqref{e:inequalitynosimiple},
 we consider the following two cases.

First, if
$$%\Big(1-\delta_{t+s}^{lb}-\frac{(1+\delta_{t}^{ub})}{a(s;ks,\alpha)}\Big)
\rho_{ks}\|\bm{h}_{T_{01}}\|_{2}
 -\frac{(1+\delta_{ks}^{ub})\sqrt{s}}{\sqrt{ks}-\alpha}\frac{2\|\bm{x}_{-\max(s)}\|_1}{\sqrt{s}}<0,$$
 i.e.,
 $$\|\bm{h}_{T_{01}}\|_{2}<
\frac{(1+\delta_{ks}^{ub})\sqrt{s}}{(\sqrt{ks}-\alpha)\rho_{ks}}\frac{2\|\bm{x}_{-\max(s)}\|_1}{\sqrt{s}}.
 $$

%Note that  the condition \eqref{RIPCondition3} implies
%From  $a(s;ks,\alpha)b(k,s)>1$, it follows that
 %\begin{align*}
 %&1-\delta_{(k+1)s}^{lb}-\frac{1+\delta_{ks}^{ub}}{a(s;\alpha,k)}\\
 %&>1-\delta_{(k+1)s}^{lb}-\frac{1+(b(k,s)+1)\delta_{ks}^{ub}}{a(s;\alpha,k)b(k,s)}
 %\end{align*}
%\big(b(k,s)+1\big)\delta_{ks}^{ub}+a(s;ks,\alpha)b(k,s)\delta_{(k+1)s}^{lb}<a(s;ks,\alpha)b(k,s)-1

%1-\delta_{(k+1)s}^{lb}-\frac{1+(b(k,s)+1)\delta_{ks}^{ub}}{a(s;\alpha,k)b(k,s)}>0
%$$
Second, if
\begin{align*}
\Big(1-&\delta_{t+s}^{lb}-\frac{1+\delta_{ks}^{ub}}{a(s;ks,\alpha)}\Big)\|\bm{h}_{T_{01}}\|_{2}\\
&\hspace*{12pt}-\frac{(1+\delta_{ks}^{ub})\sqrt{s}}{\sqrt{ks}-\alpha}\frac{2\|\bm{x}_{-\max(s)}\|_1}{\sqrt{s}}
\geq0,
\end{align*}
 which implies
 %$$1-\delta_{t+s}^{lb}-\frac{1+\delta_{t}^{ub}}{a(s;ks,\alpha)}>0,$$
 %and
 $$ \|\bm{h}_{T_{01}}\|_{2}\geq\frac{(1+\delta_{ks}^{ub})\sqrt{s}}{(\sqrt{ks}-\alpha)\rho_{ks}}
 \frac{2\|x_{-\max(s)}\|_1}{\sqrt{s}},
 $$
% where $\rho_ks=1-\delta_{t+s}^{lb}-\frac{1+\delta_{t}^{ub}}{a(s;ks,\alpha)}$.
then the inequality \eqref{e:inequalitynosimiple} is equivalent  to
\begin{align}\label{e3.13}
&\bigg(\rho_{ks}\|\bm{h}_{T_{01}}\|_{2} -\frac{(1+\delta_{ks}^{ub})\sqrt{s}}{\sqrt{ks}-\alpha}\frac{2\|\bm{x}_{-\max(s)}\|_1}{\sqrt{s}}\bigg)^2\nonumber\\
&\leq2m\sqrt{s}\eta_2\bigg(2\|\bm{h}_{T_{01}}\|_2+\frac{2\|\bm{x}_{-\max(s)}\|_1+\alpha\|\bm{h}\|_2}{\sqrt{s}}\bigg).
\end{align}
%where $(a)_+=\max\{a,0\} $ for any $a\in\mathbb{R}$.

%Taking $t=ks\in\mathbb{Z}_+$ such that $a(s;\alpha,k):=a(s,ks)=\frac{\sqrt{ks}-\alpha}{\sqrt{s}+\alpha}>1.$  Note that
%\begin{align}\label{RIPCondition1}
%\delta_{ks}^{ub}+a(s;\alpha,k)\delta_{(k+1)s}^{lb}<a(s;\alpha,k)-1
%\end{align}
%implies that
%$$
%\rho_1=1-\delta_{(k+1)s}^{lb}-\frac{(1+\delta_{ks}^{ub})}{a(s;\alpha,k)}>0.
%$$
Let $X=\|\bm{h}_{T_{01}}\|_{2}$ and $Y=\frac{2\|\bm{x}_{-\max(s)}\|_1+\alpha\|h\|_2}{\sqrt{s}}$. By
$\frac{2\|\bm{x}_{-\max(s)}\|_1}{\sqrt{s}}\leq Y$,
%then \eqref{e3.13} is equivalent  to
to guarantee  that \eqref{e3.13} holds, it suffices to show
%Then when
%$$
%\|h_{T_{01}}\|_{2}\geq\frac{(1+\delta_{ks}^{ub})\sqrt{s}}{(\sqrt{ks}-\alpha)\rho_1}\frac{2\|x_{-\max(s)}\|_1}{\sqrt{s}},
%$$
%we have
\begin{align}\label{e3.14}
&\rho_{ks}^2X^2-\bigg(\frac{2\rho_1(1+\delta_{ks}^{ub})\sqrt{s}}{\sqrt{ks}-\alpha}Y+4m\sqrt{s}\eta_2\bigg)X\nonumber\\
&-2m\sqrt{s}\eta_2 Y\leq0.
\end{align}
%Note that
For the one-variable quadratic inequality $aZ^2-bZ-c\leq 0$ with the  constants  $a,b,c>0$,
 there is the fact that
$$
Z\leq\frac{b+\sqrt{b^2+4ac}}{2a}\leq\frac{b}{a}+\sqrt{\frac{c}{a}}.
$$
Hence,
% we can get an upper bound of $X$ as follows
\begin{align}\label{e3.15}
X\leq& \frac{2\rho_{ks}\frac{(1+\delta_{ks}^{ub})\sqrt{s}}{\sqrt{ks}-\alpha}Y+4m\sqrt{s}\eta_2}{\rho_{ks}^2}
+\sqrt{\frac{2m\sqrt{s}\eta_2 \varepsilon Y}{\rho_{ks}^2\varepsilon }}\nonumber\\
\leq&\frac{2(1+\delta_{ks}^{ub})\sqrt{s}}{(\sqrt{ks}-\alpha)\rho_{ks}}Y+\frac{4m\sqrt{s}}{\rho_{ks}^2}\eta_2
+\frac{1}{2}\bigg(\frac{2m\sqrt{s}}{\rho_{ks}^2\varepsilon }\eta_2+\varepsilon Y\bigg)\nonumber\\
=&\bigg(\frac{2(1+\delta_{ks}^{ub})\sqrt{s}}{(\sqrt{ks}-\alpha)\rho_{ks}}+\frac{\varepsilon}{2}\bigg)Y
+\bigg(4+\frac{1}{\varepsilon}\bigg)\frac{m\sqrt{s}}{\rho_{ks}^2}\eta_2,
\end{align}
where $\varepsilon > 0$ is to be determined later.

 By the above discussion and
 \begin{align*}
 &\bigg(\frac{2(1+\delta_{ks}^{ub})\sqrt{s}}{(\sqrt{ks}-\alpha)\rho_{ks}}+\frac{\varepsilon}{2}\bigg)Y
+\bigg(4+\frac{1}{\varepsilon}\bigg)\frac{m\sqrt{s}}{\rho_{ks}^2}\eta_2\\
&=\bigg(\frac{2(1+\delta_{ks}^{ub})\sqrt{s}}{(\sqrt{ks}-\alpha)\rho_1}+\frac{\varepsilon}{2}\bigg)
\frac{2\|\bm{x}_{-\max(s)}\|_1+\alpha\|\bm{h}\|_2}{\sqrt{s}}\\
 &\hspace*{12pt} +\bigg(4+\frac{1}{\varepsilon}\bigg)\frac{m\sqrt{s}}{\rho_1^2}\eta_2\\
 &\geq
 \frac{(1+\delta_{ks}^{ub})\sqrt{s}}{(\sqrt{ks}-\alpha)\rho_{ks}}
\frac{2\|\bm{x}_{-\max(s)}\|_1}{\sqrt{s}},
\end{align*}
%Therefore, %we obtain %a upper bound $\|\bm{h}_{T_{01}}\|_{2}$
%Hence whenever
%$$
%\|h_{T_{01}}\|_{2}\geq\frac{(1+\delta_{ks}^{ub})\sqrt{s}}{(\sqrt{ks}-\alpha)\rho_1}\frac{2\|x_{-\max(s)}\|_1}{\sqrt{s}}
%$$
%or not, we always have
the inequality \eqref{e3.13} always holds
when
\begin{align}\label{e3.16}
\|\bm{h}_{T_{01}}\|_{2}
%\leq \max\Bigg\{\frac{(1+\delta_{ks}^{ub})\sqrt{s}}{(\sqrt{ks}-\alpha)\rho_1}\frac{2\|x_{-\max(s)}\|_1}{\sqrt{s}},\nonumber\\
%&\bigg(\frac{2(1+\delta_{ks}^{ub})\sqrt{s}}{(\sqrt{ks}-\alpha)\rho_1}+\frac{\varepsilon}{2}\bigg)
%\frac{2\|x_{-\max(s)}\|_1+\alpha\|h\|_2}{\sqrt{s}}
%+\bigg(4+\frac{1}{\varepsilon}\bigg)\frac{m\sqrt{s}}{\rho_1^2}\eta_2\Bigg\}\nonumber\\
\leq&\bigg(\frac{2(1+\delta_{ks}^{ub})\sqrt{s}}{(\sqrt{ks}-\alpha)\rho_{ks}}+\frac{\varepsilon}{2}\bigg)
\frac{2\|\bm{x}_{-\max(s)}\|_1+\alpha\|\bm{h}\|_2}{\sqrt{s}}\nonumber\\
&+\bigg(4+\frac{1}{\varepsilon}\bigg)\frac{m\sqrt{s}}{\rho_{ks}^2}\eta_2,
\end{align}
which presents an upper bound $\|\bm{h}_{T_{01}}\|_{2}$.

Next, we will estimate $\|\bm{h}_{T_{01}^c}\|_{2}$.
In terms of the derivations of (\ref{e3.9}) and
\eqref{e:hupperbound}, they  still hold.
%therefore we have
%\begin{align}\label{e3.17}
%\|h\|_2
%&\leq\bigg(1+\frac{1}{2\sqrt{k}}\bigg)\|h_{T_{01}}\|_2+\frac{1}{2\sqrt{k}}\frac{2\|x_{-\max(s)}\|_1}{\sqrt{s}}
%+\frac{\alpha}{2\sqrt{k}}\frac{\|h\|_2}{\sqrt{s}}
%\end{align}
%instead of (\ref{e3.10}).

Substituting (\ref{e3.16}) into (\ref{e:hupperbound}), ones obtain
\begin{align*}
&\|\bm{h}\|_2
\leq\bigg(1+\frac{1}{2\sqrt{k}}\bigg)
\Bigg(
\bigg(4+\frac{1}{\varepsilon}\bigg)\frac{m\sqrt{s}}{\rho_{ks}^2}\eta_2\\
&\hspace*{12pt}+\bigg(\frac{2(1+\delta_{ks}^{ub})\sqrt{s}}{(\sqrt{ks}-\alpha)\rho_{ks}}+\frac{\varepsilon}{2}\bigg)
\frac{2\|\bm{x}_{-\max(s)}\|_1+\alpha\|\bm{h}\|_2}{\sqrt{s}}\Bigg)\\
&\hspace*{12pt}+\frac{1}{2\sqrt{k}}\frac{2\|\bm{x}_{-\max(s)}\|_1}{\sqrt{s}}
+\frac{\alpha}{2\sqrt{k}}\frac{\|\bm{h}\|_2}{\sqrt{s}}\\
&\leq\frac{1}{2\sqrt{k}}\Bigg((2\sqrt{k}+1)\bigg(\frac{2(\sqrt{s}+\alpha)(1+\delta_{ks}^{ub})}
{(\sqrt{ks}-\alpha)\rho_{ks}}
+\frac{\varepsilon}{2}\bigg)+1\Bigg)\\
&\hspace*{24pt}\times\frac{2\|\bm{x}_{-\max(s)}\|_1}{\sqrt{s}}\\
&\hspace*{12pt}+\frac{1}{2\sqrt{k}}
\Bigg((2\sqrt{k}+1)\bigg(\frac{2(\sqrt{s}+\alpha)(1+\delta_{ks}^{ub})}{(\sqrt{ks}-1)\rho_{ks}}
+\frac{\varepsilon}{2}\bigg)+1\Bigg)\\
&\hspace*{24pt}\times\frac{\alpha\|\bm{h}\|_2}{\sqrt{s}}\\
&\hspace*{12pt}+\bigg(1+\frac{1}{2\sqrt{k}}\bigg)\bigg(4+\frac{1}{\varepsilon}\bigg)\frac{m\sqrt{s}}{\rho_{ks}^2}\eta_2\\
&=\varrho\frac{2\|\bm{x}_{-\max(s)}\|_1}{\sqrt{s}}+\frac{\alpha\varrho}{\sqrt{s}}\|\bm{h}\|_2\\
&\hspace*{12pt}+\bigg(4+\frac{1}{\varepsilon}\bigg)\frac{(2\sqrt{k}+1)m\sqrt{s}}{2\sqrt{k}\rho_1^2}\eta_2,
\end{align*}
where the last equality is from
$$
\varrho=\frac{1}{2\sqrt{k}}\Bigg((2\sqrt{k}+1)\bigg(\frac{2(1+\delta_{ks}^{ub})}{a(s,ks;\alpha)\rho_{ks}}
+\frac{\varepsilon}{2}\bigg)+1\Bigg).
$$

Taking
$$
\varepsilon=\frac{2(1+\delta_{ks}^{ub})}{8a(s,ks;\alpha)\rho_{ks}},
$$
then
\begin{align}
\varrho=\frac{1}{2\sqrt{k}}\Bigg((2\sqrt{k}+1)\frac{17(1+\delta_{ks}^{ub})}{16a(s,ks;\alpha)\rho_{ks}}+1\Bigg)
<\frac{\sqrt{s}}{\alpha},
\end{align}
%And taking $k>0$ such that $ks\in\mathbb{Z}_+$ with
%\begin{align*}
%b(s;\alpha,k)&:=\frac{(2\sqrt{ks}-\alpha)}{3(2\sqrt{k}+1)}>0
%\end{align*}
%satisfying $(a(s;\alpha,k)-1)b(s;\alpha,k)>0$.
where the last inequality is from \eqref{RIPCondition3}.
In fact,
\begin{align*}
&\varrho-\frac{\sqrt{s}}{\alpha}
=\frac{1}{2\sqrt{k}}\Bigg((2\sqrt{k}+1)\frac{3(1+\delta_{ks}^{ub})}{a(s,ks;\alpha)\rho_{ks}}+1\Bigg)
-\frac{\sqrt{s}}{\alpha}\\
&=\frac{17(2\sqrt{k}+1)}{16\sqrt{k}a(s,ks;\alpha)\rho_{ks}}\\
&\hspace*{12pt}\Bigg(1+\delta_{ks}^{ub}-
\bigg(1-\frac{2\sqrt{ks}}{\alpha}\bigg)\frac{8}{17(2\sqrt{k}+1)}a(s,ks;\alpha)\rho_{ks}\Bigg)\\
&=:\frac{17(2\sqrt{k}+1)}{16\sqrt{k}a(s,ks;\alpha)\rho_{ks}}
\Bigg(1+\delta_{ks}^{ub}-a(s,ks;\alpha)b(s,k;\alpha)\rho_{ks}\Bigg),
\end{align*}
where
$$
b(s,k;\alpha)=\bigg(1-\frac{2\sqrt{ks}}{\alpha}\bigg)\frac{8}{17(2\sqrt{k}+1)}
=\frac{8(2\sqrt{ks}-\alpha)}{17\alpha(2\sqrt{k}+1)}.
$$
Then,
\begin{align*}
\varrho-\frac{\sqrt{s}}{\alpha}&=\frac{17(2\sqrt{k}+1)}{16\sqrt{k}a(s,ks;\alpha)\rho_{ks}}\\
&\hspace*{12pt}\times\bigg((b(s,k;\alpha)+1)\delta_{ks}^{ub}+a(s,ks;\alpha)b(s,k;\alpha)\delta_{(k+1)s}^{ub}\\
 & \ \ \ \ -\Big(a(s,ks;\alpha)b(s,k;\alpha)-b(s,k;\alpha)-1\Big)\bigg)\\
&<0
\end{align*}
where the equality is from  the definition of $\rho_{ks}$.
Therefore
\begin{align*}
\|\bm{h}\|_2&\leq\frac{\sqrt{s}\varrho}{\sqrt{s}-\alpha\varrho}\frac{2\|\textbf{x}_{-\max(s)}\|_1}{\sqrt{s}}\\
&\hspace*{12pt}+\frac{2(2\sqrt{k}+1)\big((1+\delta_{ks}^{ub})+a(s,ks;\alpha)\rho_{ks}\big)ms}
{\sqrt{k}(\sqrt{s}-\alpha\varrho)(1+\delta_{ks}^{up})\rho_{ks}^2}\eta_2,
\end{align*}
which finishes the proof of Theorem \ref{StableRecoveryviaVectorL1-alphaL2-DS}.
\end{proof}

\hskip\parindent

\textbf{Acknowledgement}: Peng Li would like to thank Dr. Jiaxin Xie (Academy of Mathematics and Systems Science, CAS) and Meng Huang (Academy of Mathematics and Systems Science, CAS) for some discussions and suggestions about the $\ell_{1-2}$ minimization.  Peng Li also thanks Jingjing Liu (Graduate School, China Academy of Engineering Physics) for her discussion about image process. Wengu Chen is supported by Natural Science Foundation of China (No. 11871109) and NSAF (Grant No.U1830107).

\end{document}